\documentclass[12pt,reqno]{amsart}

\usepackage{amsmath , amssymb, latexsym }
\usepackage{graphics}

\usepackage[mathscr]{eucal}

\newdimen\unit\newdimen\psep\newcount\nd\newcount\ndx\newbox\dotb\newbox\ptbox
\newdimen\dx\newdimen\dy\newdimen\dxx\newdimen\dyy\newdimen\hgt
\newdimen\xoff\newdimen\yoff
\newcommand\clap[1]{\hbox to 0pt{\hss{#1}\hss}}
\newcommand\vdisk[1]{{\font\dotf=cmr10 scaled #1\dotf.}}
\newcommand\varline[2]{\setbox\dotb\hbox{\vdisk{#1}}\xoff=-.5\wd\dotb
\wd\dotb=0pt\yoff=-.5\ht\dotb\psep=#2\ht\dotb}
\newcommand\varpt[1]{\setbox\ptbox\clap{\vdisk{#1}}\setbox\ptbox
\hbox{\raise-.5\ht\ptbox\box\ptbox}}
\newcommand\cpt{\copy\ptbox}
\newcommand\point[3]{\rlap{\kern#1\unit\raise#2\unit\hbox{#3}}}
\newcommand\setnd[4]{\dx=#3\unit\advance\dx-#1\unit\divide\dx by\psep
\dy=#4\unit\advance\dy-#2\unit\divide\dy by\psep \multiply\dx
by\dx\multiply\dy by\dy\advance\dx\dy\nd=1\advance\dx-1sp
\loop\ifnum\dx>0\advance\dx-\nd sp\advance\nd1\advance\dx-\nd
sp\repeat}

\newcommand\dline[5]{{\nd=#5\hgt=#2\unit\dx=#3\unit\advance\dx-#1\unit
\divide\dx by\nd\dy=#4\unit\advance\dy-#2\unit\divide\dy by\nd
\advance\hgt\yoff\rlap{\kern#1\unit\kern\xoff\loop\ifnum\nd>1\advance\nd-1
\advance\hgt\dy\kern\dx\raise\hgt\copy\dotb\repeat}}}

\newcommand\qellip[4]{{\setnd{0}{0}{#3}{#4}\dx=\unit\dy=0pt\raise\yoff\rlap{%
\kern#1\unit\kern\xoff\raise#2\unit\hbox{\loop\ifnum\dx>0\rlap{\kern#3\dx
\raise#4\dy\copy\dotb}\hgt=\dx\divide\hgt
by\nd\advance\dy\hgt\hgt=\dy \divide\hgt
by\nd\advance\dx-\hgt\repeat\rlap{\raise#4\dy\copy\dotb}}}}}
\newcommand\bez[6]{{\setnd{#1}{#2}{#3}{#4}\ndx=\nd\setnd{#3}{#4}{#5}{#6}
\ifnum\ndx>\nd\nd=\ndx\fi\dx=#3\unit\advance\dx-#1\unit\dy=#4\unit
\advance\dy-#2\unit\dxx=#5\unit\advance\dxx-#1\unit\dyy=#6\unit\advance
\dyy-#2\unit\advance\dxx-2\dx\advance\dyy-2\dy\divide\dxx
by\nd\divide\dyy
by\nd\advance\dx.25\dxx\advance\dy.25\dyy\divide\dx
by\nd\divide\dy by\nd \multiply\nd
by2\dx=100\dx\dy=100\dy\dxx=100\dxx\dyy=100\dyy\divide\dxx by\nd
\divide\dyy
by\nd\hgt=#2\unit\raise\yoff\rlap{\kern#1\unit\kern\xoff
\raise\hgt\copy\dotb\loop\ifnum\nd>0\advance\nd-1\advance\hgt0.01\dy
\kern0.01\dx\raise\hgt\copy\dotb\advance\dx\dxx\advance\dy\dyy\repeat}}}
\newcommand\ptu[3]{\point{#1}{#2}{\cpt\raise1ex\clap{$\scriptstyle{#3}$}}}
\newcommand\ptd[3]{\point{#1}{#2}{\cpt\raise-1.8ex\clap{$\scriptstyle{#3}$}}}
\newcommand\ptr[3]{\point{#1}{#2}{\cpt\raise-.4ex\rlap{$\ \scriptstyle{#3}$}}}
\newcommand\ptl[3]{\point{#1}{#2}{\cpt\raise-.4ex\llap{$\scriptstyle{#3}\ $}}}
\newcommand\ptlu[3]{\point{#1}{#2}{\raise.8ex\clap{$\scriptstyle{#3}$}}}
\newcommand\ptld[3]{\point{#1}{#2}{\raise-1.6ex\clap{$\scriptstyle{#3}$}}}
\newcommand\ptlr[3]{\point{#1}{#2}{\raise-.4ex\rlap{$\,\scriptstyle{#3}$}}}
\newcommand\ptll[3]{\point{#1}{#2}{\raise-.4ex\llap{$\scriptstyle{#3}\,$}}}

\newcommand\thnline{\varline{400}{.6}}

\varpt{2500}\thnline\unit=2em

\newtheorem{thm}{Theorem}
\newtheorem*{SRL}{Szemer\'edi's Regularity Lemma}
\newtheorem*{RSRL}{Szemer\'edi's Lemma: refinement version}
\newtheorem*{ESStab}{Erd\H{o}s-Simonovits Stability Theorem}
\newtheorem*{Sauer}{Sauer's Lemma}
\newtheorem*{embed}{Embedding Lemma}
\newtheorem*{slice}{The Slicing Lemma}
\newtheorem*{BT}{The Alekseev-Bollob\'as-Thomason Theorem}
\newtheorem*{BBS}{The Balogh-Bollob\'as-Simonovits Theorem}
\newtheorem{conj}{Conjecture}
\newtheorem{qu}{Question}
\newtheorem{prob}{Problem}
\newtheorem{lemma}[thm]{Lemma}

\newtheorem{cor}[thm]{Corollary}

\newtheorem{obs}[thm]{Observation}
\theoremstyle{definition}\newtheorem{rmk}{Remark}
\theoremstyle{definition}\newtheorem*{defn}{Definition}
\theoremstyle{definition}\newtheorem*{alg}{Algorithm}
\theoremstyle{definition}
\newcommand{\ds}{\displaystyle}
\newcommand{\ul}{\underline}
\newcommand{\ol}{\overline}

\def\A{\mathcal{A}}
\def\B{\mathcal{B}}
\def\C{\mathcal{C}}
\def\D{\mathcal{D}}

\def\G{\mathcal{G}}
\def\HH{\mathcal{H}}
\def\I{\mathcal{I}}
\def\J{\mathcal{J}}
\def\K{\mathcal{K}}
\def\M{\mathcal{M}}

\def\P{\mathcal{P}}

\def\U{\mathcal{U}}

\def\Ex{\mathbb{E}}

\def\N{\mathbb{N}}
\def\Pr{\mathbb{P}}

\def\le{\leqslant}
\def\ge{\geqslant}

\def\eps{\varepsilon}
\def\->{\rightarrow}
\def\<{\langle}
\def\>{\rangle}

\def\ex{\textup{ex}}
\def\ent{\textup{ent}}

\begin{document}
\title{The structure of almost all graphs in a hereditary property}

\author{Noga Alon}
\address{Schools of Mathematics and Computer Science, Raymond and Beverly Sackler Faculty of Exact Sciences, Tel Aviv
University, Tel Aviv 69978, Israel and IAS, Princeton, NJ 08540, USA} \email{nogaa@tau.ac.il}

\author{J\'ozsef Balogh}
\address{Department of Mathematics\\ University of Illinois\\ 1409 W. Green Street\\ Urbana, IL 61801, USA} \email{jobal@math.uiuc.edu}

\author{B\'ela Bollob\'as}
\address{Trinity College\\ Cambridge CB2 1TQ\\ England\\ and \\ Department of Mathematical Sciences\\ The University of Memphis\\ Memphis, TN 38152, USA} \email{B.Bollobas@dpmms.cam.ac.uk}

\author{Robert Morris}
\address{Murray Edwards College, The University of Cambridge, Cambridge CB3 0DF, England} \email{rdm30@cam.ac.uk}
\thanks{The first author was supported in part by a USA Israeli BSF grant, by a grant from the Israel Science Foundation, by an ERC advanced grant, by NSF grant CCF 0832797 and by the Ambrose Monell Foundation. The second author was supported by NSF CAREER Grant DMS-0745185 and DMS-0600303, UIUC Campus Research Board Grants 09072 and 08086, and OTKA Grant K76099. The third author was supported by NSF grants DMS-0505550, CNS-0721983 and CCF-0728928, and ARO grant W911NF-06-1-0076. The fourth author was partly supported by a JSPS fellowship.}

\begin{abstract}
A hereditary property of graphs is a collection of graphs which is closed under taking induced subgraphs. The speed of $\P$ is the function $n \mapsto |\P_n|$, where $\P_n$ denotes the graphs of order $n$ in $\P$. It was shown by Alekseev, and by Bollob\'as and Thomason, that if $\P$ is a hereditary property of graphs then
$$|\P_n| \; = \; 2^{(1 - 1/r + o(1))n^2/2},$$
where $r = r(\P) \in \N$ is the so-called `colouring number' of $\P$. However, their results tell us very little about the structure of a typical graph $G \in \P$.

In this paper we describe the structure of almost every graph in a hereditary property of graphs, $\P$. As a consequence, we derive essentially optimal bounds on the speed of $\P$, improving the Alekseev-Bollob\'as-Thomason Theorem, and also generalizing results of Balogh, Bollob\'as and Simonovits.
\end{abstract}

\maketitle

\section{Introduction}\label{intro}

In this paper we shall describe the structure of almost every graph in an arbitrary hereditary property of graphs, $\P$. As a corollary, we shall obtain bounds on the speed of $\P$ which improve those proved by Alekseev~\cite{Alek} and Bollob\'as and Thomason~\cite{BT1,BT2}, and generalize a theorem of Balogh, Bollob\'as and Simonovits~\cite{BBS1,BBS2} on monotone properties of graphs. We begin with some definitions.

A collection of labelled graphs, $\P$, is called a \emph{hereditary property} if it is closed under re-labelling vertices, and under taking induced subgraphs. It is called \emph{monotone} if it is moreover closed under taking arbitrary subgraphs. Note that a hereditary property may be characterized by a (possibly infinite) collection of forbidden induced subgraphs.

Given a property of graphs, $\P$, let $\P_n = \{G \in \P : |V(G)| = n\}$ denote the graphs in $\P$ with vertex set $[n]$. The \emph{speed} of $\P$, introduced in 1976 by Erd\H{o}s, Kleitman and Rothschild~\cite{EKR}, is the function
$$n \; \mapsto \; |\P_n|.$$
The speed is a natural measure of the `size' of a property.

The possible structures and speeds of a hereditary or monotone property of graphs have been extensively studied, originally in the special case where a single subgraph is forbidden, and more recently in general. For example, Erd\H{o}s, Kleitman and Rothschild~\cite{EKR} and Kolaitis, Pr\"omel and Rothschild~\cite{KPR} studied $K_r$-free graphs, Erd\H{o}s, Frankl and R\"odl~\cite{EFR} studied monotone properties when a single graph is forbidden, and Pr\"omel and Steger~\cite{PS1,PS6} obtained (amongst other things) very precise results on the structure of almost all (induced-)$C_4$-free and $C_5$-free graphs. They also were the first to define the following parameter of a property of graphs, known as the `colouring number' of $\P$, which will be important in what follows.

First, for each $r \in \N$ and each vector $v \in \{0,1\}^r$, define a collection $\HH(r,v)$ of graphs as follows. Let $G \in \HH(r,v)$ if $V(G)$ may be partitioned into $r$ sets $A_1, \ldots, A_r$ such that $G[A_j]$ is the empty graph if $v_j = 0$, and is the complete graph if $v_j = 1$.

\begin{defn}
The colouring number $\chi_c(\P)$ of a property of graphs, $\P$, is defined to be
$$\chi_c(\P) \; := \; \max\big\{ r \in \N \,:\, \HH(r,v) \subset \P \textup{ for some } v \in \{0,1\}^r \big\}.$$
\end{defn}

The following result, proved by Alekseev~\cite{Alek} and Bollob\'as and Thomason~\cite{BT1,BT2}, generalizes the Erd\H{o}s-Frankl-R\"odl Theorem to a general hereditary property of graphs.

\begin{BT}
Let $\P$ be a hereditary property of graphs, and suppose $\chi_c(\P) = r$. Then
$$|\P_n| \; = \; 2^{(1 - 1/r + o(1))n^2/2}.$$
\end{BT}

The Alekseev-Bollob\'as-Thomason Theorem shows that the set of possible values for the `entropy' of a hereditary property of graphs,
$$\textup{ent}(\P) \; := \; \lim_{n \to \infty} \frac{1} {{n \choose 2}} \log_2(|\P_n|) $$
is not continuous, but in fact undergoes a series of discrete `jumps', from $1 - \frac{1}{r}$ to $1 - \frac{1}{r+1}$, where $r \in \N$. However, the proofs of Alekseev and of Bollob\'as and Thomason tell us very little about the structure of a \emph{typical} graph $G \in \P$. Their theorem also gives rather weak bounds on the rate of convergence of the entropy as $n \to \infty$.

For monotone properties of graphs, these problems were addressed by Balogh, Bollob\'as and Simonovits~\cite{BBS1,BBS2,BBS3} in a series of papers in which they proved \emph{very} precise structural results, and obtaining correspondingly precise bounds on the rate of convergence of $\ent(\P)$.

The following theorem was the main result of~\cite{BBS1}. Let $\ex(n,\M)$ denote the usual extremal number of a collection of graphs $\M$.

\begin{BBS}
Let $\P$ be a monotone property of graphs, with colouring number $\chi_c(\P) = r$. Then there exists a family of graphs $\M$ (containing a bipartite graph), and a constant $c = c(\P)$, such that the following holds.

For almost all graphs $G \in \P$, there exists a partition $(A,S_1,\ldots,S_r)$ of $V(G)$, such that
\begin{itemize}
\item[$(a)$] $|A| \le c(\P)$, \\[-1.5ex]
\item[$(b)$] $G[S_j]$ is $\M$-free for every $j \in [r]$,
\end{itemize}
and moreover
$$2^{(1-1/r){n \choose 2}} n^{\ex(n/r,\M)} \; \le \; |\P_n| \; \le \; 2^{(1 - 1/r){n \choose 2}} n^{\ex(n,\M) + cn}.$$
\end{BBS}

For even more precise results see~\cite{BBS2} and ~\cite{BBS3}. Balogh, Bollob\'as and Simonovits also had the following (unpublished) conjecture regarding the speed of hereditary graph properties. Let $\P^i(n,\M)$ denote the collection of induced-$\M$-free graphs on $[n]$.

\begin{conj}[Balogh, Bollob\'as, Simonovits]\label{herconj}
Let $\P$ be a hereditary property of graphs, and suppose that $\chi_c(\P) = r$. Then there exists a family of graphs $\M$ (with $\chi_c(\M) = 1$), and a constant $c = c(\P)$, such that
$$2^{(1-1/r){n \choose 2}} 2^{|\P^i(n/r,\M)|} \; \le \; |\P_n| \; \le \; 2^{(1-1/r){n \choose 2}} n^{r|\P^i(n/r,\M)| + cn}.$$
\end{conj}

We remark that, by the results of Pr\"omel and Steger~\cite{PS1,PS6}, Conjecture~\ref{herconj} holds for the properties $\P = \{G : C_4 \not\le G\}$ and $\P = \{G : C_5 \not\le G\}$. (Here, and throughout, we write $H \le G$ to mean that $H$ is an induced subgraph of $G$.) To be precise, they proved that the vertex set of almost all $C_4$-free graphs can be partitioned into an independent set and a complete graph, and that for almost every $C_5$-free graph $G$, either $G$ or $\ol{G}$ has the following structure: the vertex set may be partitioned into two classes, $V(G) = A \cup B$, so that $A$ induces a clique, and $B$ induces a vertex disjoint union of cliques.

Finally, we note that even more precise structural results have been obtained for hereditary properties of graphs with much lower speeds, by Balogh, Bollob\'as and Weinreich~\cite{BBW1,BBW2,BBW3}, and for hereditary properties of other combinatorial structures, see for example~\cite{AFN,BBM1,BBM2,BBSS,Klaz,MT}. Note in particular~\cite{AFN}, where Sauer's Lemma (which will be a crucial tool in our proof) is used. There has also been some important recent progress on hereditary properties of hypergraphs, by Dotson and Nagle~\cite{DN} and (independently) Ishigami~\cite{Ish}, who (building on work of Nagle and R\"odl~\cite{NR} and Kohayakawa, Nagle and R\"odl~\cite{KNR}) proved a version of the Alekseev-Bollob\'as-Thomason Theorem for $k$-uniform hypergraphs, and by Person and Schacht~\cite{PSch}, who showed that almost every Fano-plane-free 3-uniform hypergraph is bipartite. 

\section{Main Results}\label{resultsec}

In this section we state our main results. We begin with a definition.

\begin{defn}
For each $k \in \N$, the \emph{universal graph} $U(k)$ is the bipartite graph with parts $A \cong [2]^k$ and $B \cong [k]$, and edge set
$$E\big( U(k) \big) \; = \; \big\{ ab \, : \, a \in A, b \in B \textup{ and } b \in a \big\}.$$
A graph $G$ is said to be \emph{$U(k)$-free} if there do not exist disjoint subsets $A,B \subset V(G)$ such that $G[A,B] = U(k)$. ($G[A,B]$ denotes the bipartite graph induced by the pair $(A,B)$.)
\end{defn}

The following theorem is the main result of this paper.

\begin{thm}\label{structure}
Let $\P$ be a hereditary property of graphs, with colouring number $\chi_c(\P) = r$. Then there exist constants $k = k(\P) \in \N$ and $\eps = \eps(\P) > 0$ such that the following holds.

For almost all graphs $G \in \P$, there exists a partition $(A,S_1,\ldots,S_r)$ of $V(G)$, such that
\begin{itemize}
\item[$(a)$] $|A| \le n^{1-\eps}$, \\[-1.5ex]
\item[$(b)$] $G[S_j]$ is $U(k)$-free for every $j \in [r]$.
\end{itemize}
Moreover
$$2^{(1 - 1/r)n^2/2} \; \le \; |\P_n| \; \le \; 2^{(1 - 1/r)n^2/2 \,+\, n^{2-\eps}}$$
for every sufficiently large $n \in \N$.
\end{thm}

We shall in fact prove, not just that this structural description holds for almost all graphs $G \in \P$, but that the number of graphs in $\P_n$ for which it fails is at most
$$2^{-n^{2-\eps}}|\P_n|$$
if $n$ is sufficiently large. The final part of Theorem~\ref{structure} is an immediate consequence of conditions $(a)$ and $(b)$, and the following theorem.

\begin{thm}\label{count}
For each $k \in \N$ there exists $\eps = \eps(k) > 0$ such that the following holds. There are at most $2^{n^{2-\eps}}$ distinct $U(k)$-free graphs on $[n]$.
\end{thm}

The structure of the remainder of the paper is as follows. First, in Section~\ref{toolsec}, we shall state the main tools we shall use in the paper: these include the Szemer\'edi Regularity Lemma, the Erd\H{o}s-Simonovits Stability Theorem, and Sauer's Lemma. In Section~\ref{sketchsec} we give a sketch of the proof of our main result in the case $\chi_c(\P) = 2$, and in Section~\ref{T2sec} we prove Theorem~\ref{count}, and deduce our bounds on the speed of $\P$. In Section~\ref{lemmasec} we prove various lemmas on $U(k)$-free graphs, and in Section~\ref{proofsec} we prove Theorem~\ref{structure}. In Section~\ref{bipsec} we shall show how to prove even sharper results in the bipartite case, and in Section~\ref{qsec}, we finish by stating a couple of questions and open problems.

\section{tools}\label{toolsec}

In this section we shall recall some of the important tools we shall use in order to prove Theorems~\ref{structure} and~\ref{count}. In particular, we shall recall the Szemer\'edi Regularity Lemma~\cite{Sz}, the Erd\H{o}s-Simonovits Stability Theorem~\cite{ES1,ES2}, and Sauer's Lemma~\cite{Sauer}.

Let $G$ be a graph, let $A,B \subset V(G)$ with $A \cap B = \emptyset$, and let $\eps > 0$. We write $d(A,B)$ for the density of the bipartite graph $G[A,B]$. We say that the pair $(A,B)$ is $\eps$-regular if
$$|d(A,B) - d(X,Y)| < \eps$$
for every $X \subset A$ and $Y \subset B$ with $|X| \ge \eps|A|$ and $|Y| \ge \eps|B|$.

\begin{defn}
A partition $A_1 \cup \ldots \cup A_k$ of $V(G)$ is said to be a \emph{Szemer\'edi partition} of $G$ for $\eps$ if $|A_1| \le \ldots \le |A_k| \le |A_1| + 1$, and all but $\eps k^2$ of the pairs $(A_i,A_j)$ are $\eps$-regular.
\end{defn}

\begin{SRL}[Szemer\'edi, 1976]
Let $\eps > 0$ and $m \in \N$. There exists an $M = M(m,\eps) \in \N$ such that, given any graph $G$, there exists a Szemer\'edi partition of $G$ for $\eps$ into $k$ parts, for some $m \le k \le M$.
\end{SRL}

We shall also need the following `refinement' version of Szemer\'edi's Lemma. Let $G$ be a graph and $P = (U_1, \ldots, U_t)$ be a partition of $V(G)$. A \emph{Szemer\'edi refinement} of the partition $P$ for $\eps$ is a refinement of the partition $P$ which is also a Szemer\'edi partition of $G$ for $\eps$.

\begin{RSRL}
Let $\eps > 0$ and $m \in \N$. There exists an $M = M(m,\eps) \in \N$ such that, given any graph $G$, and any partition $P$ of $V(G)$ into at most $m$ parts, there exists a Szemer\'edi refinement of $P$ for $\eps$ into $k$ parts, for some $m \le k \le M$.
\end{RSRL}

Let $T_r(n)$ denote the Tur\'an graph, and $t_r(n) = e(T_r(n))$ the Tur\'an number, as usual.

\begin{ESStab}[Erd\H{o}s, Simonovits, 1968]
For each $r \in \N$ and $\eps > 0$, there exists a $\delta > 0$ such that the following holds. Let $G$ be a graph, and suppose that $K_{r+1} \not\subset G$, but
$$e(G) \, \ge \, t_r(n) \, - \, \delta n^2.$$
Then we can change $G$ into $T_r(n)$ by switching at most $\eps n^2$ edges.
\end{ESStab}

We say a set $X \subset [n]$ is shattered by a family of subsets $\A \subset \P(n)$ if for every set $B \subset X$, there exists an $A \in \A$ such that $B \cap X = A$.

\begin{Sauer}[Sauer, 1972]
Let $\A \subset \P(n)$, and suppose that
$$|\A| \; > \; \sum_{i=0}^{k-1} {n \choose i}.$$
Then there exists a $k$-set $X \subset [n]$ which is shattered by $\A$.
\end{Sauer}

Given $A,B \subset V(G)$, we shall also write $A \to B$ to mean $A$ `shatters' $B$, i.e., that $G[A',B] = U(k)$ for some $A' \subset A$, where $k = |B|$. Note that $A \to B$ if and only if the family of sets $\A = \{\Gamma(v) \cap B : v \in A\}$ shatters $B$.

Given $\eps,\delta > 0$, a pair $(A,B)$ of subsets $A,B \subset V(G)$ is said to be $(\eps,\delta)$-grey if the graph $G[A,B]$ is $\eps$-regular and has density $\delta \le d(A,B) \le 1 - \delta$.

The following Embedding Lemma is well-known (see~\cite{KS}, Theorem 2.1).

\begin{embed}
Let $\delta > 0$, and let $H$ be a graph on $r$ vertices. There exist $\eps > 0$ and $n_0 = n_0(\eps,\delta,r) \in \N$ such that the following holds.

Let $G$ be a graph on vertex set $A_1 \cup \ldots \cup A_r$, where $|A_j| \ge n_0$ for each $j \in [r]$. Then there exist distinct vertices $v_k \in A_k$ for each $k \in [r]$ such that, whenever the pair $(A_i,A_j)$ is $(\eps,\delta)$-grey, we have
$$v_iv_j \in E(G) \; \Leftrightarrow \; ij \in E(H).$$
\end{embed}

We shall also use the following simple result, known as the Slicing Lemma (see~\cite{KS}, Fact 1.5).

\begin{slice}
If $(A,B)$ is $(\eps,\delta)$-grey and $X \subset A$, $Y \subset B$ with $|X| \ge \alpha|A|$ and $|Y| \ge \alpha|B|$, then the pair $(X,Y)$ is $(\eps',\delta')$-grey, where $\eps' = 2\eps/\alpha$ and $\delta' = \delta/\alpha - \eps$.
\end{slice}

Finally, we make a trivial observation.

\begin{obs}\label{ABT}
Let $\P$ be a hereditary property of graphs, and suppose $\chi_c(\P) = r$. Then
$$|\P_n| \; \ge \; 2^{(1 - 1/r + o(1))n^2/2}.$$
\end{obs}

\begin{proof}
By definition: there are this many graphs in $\HH(r,v)$.
\end{proof}

As noted in the introduction, Alekseev~\cite{Alek} and Bollob\'as and Thomason~\cite{BT1,BT2} independently proved the corresponding upper bound.

\section{Sketch of the proof}\label{sketchsec}

Before proving Theorem~\ref{structure}, let us a give a brief (and imprecise) sketch of the proof. For simplicity we shall only consider the case $\chi_c(\P) = 2$.

Let $\eps,\delta,\gamma,\alpha > 0$ be sufficiently small, with $\eps \ll \delta \ll \gamma \ll \alpha$, and let $G \in \P_n$, where $n$ is sufficiently large. We shall say that the bipartite subgraph of $G$ induced by a pair of sets is \emph{grey} if it is $(\eps,\delta)$-grey, i.e., it is $\eps$-regular and of density between $\delta$ and $1-\delta$. We shall also use the following definition of a \emph{generalized universal graph}:
\begin{align*}
& \textup{Let $U(3,k)$ denote the $3$-partite graph on vertex set $A \cup B \cup C$, where} \\
& \textup{$|A| = 2^{|B| + |C|}$, $|B| = 2^k$ and $|C| = k$, such that $B \to C$ and $A \to B \cup C$.}
\end{align*}

The first step in the proof is to show that, for almost all $G \in \P$, there is a partition $(S_1,S_2)$ of $V(G)$ such that,
\begin{itemize}
\item[$(a)$] Each part is a union of (an almost equal number of) Szemer\'edi sets.
\item[$(b)$] Each part contains at most $\gamma m^2$ grey pairs (where $m$ is the total number of Szemer\'edi sets).
\end{itemize}
The proof of this (see Lemma~\ref{S1S2}) follows as in~\cite{BBS1}, by applying the Szemer\'edi Regularity Lemma and Erd\H{o}s-Simonovits Stability Theorem. The key observation is that the `cluster graph' contains no grey triangles (see Lemma~\ref{embed2}). We call such a partition $(S_1,S_2)$ a \emph{BBS-partition} of $G$.

Next we consider a maximal set $B \subset V(G)$ of vertices such that, for each $j = 1,2$, and each pair $b,b' \in B$,
$$\left| \big( \Gamma(b) \cap S_j \big) \triangle \big( \Gamma(b') \cap S_j \big) \right| \; \ge \; \alpha n.$$
We shall sometimes refer to a set with this property as a set of `bad' vertices. The main step in the proof is to show that, for almost every $G \in \P$, $|B|$ is bounded.

Indeed, we show (see Lemmas~\ref{key}, \ref{key2} and~\ref{key3}) that, for any $t \in \N$, if $|B|$ is sufficiently large then there exist a set $B' \subset B$ of size $2^{2t}$, and sets $T^{(1)}_1, \ldots, T^{(1)}_t \subset S_1$ and $T^{(2)}_1, \ldots, T^{(2)}_t \subset S_2$, with $|T^{(i)}_j| \ge \delta n$, such that the following holds:
\begin{itemize}
\item[$(a)$] All vertices of $T^{(i)}_j$ have the same neighbourhood in $B'$.
\item[$(b)$] If $b_1,b_2 \in B'$ with $b_1 \neq b_2$, then $\Gamma(b_1) \cap \bigcup T^{(i)}_j \neq \Gamma(b_2) \cap \bigcup T^{(i)}_j$. (Since there are $2^{2t}$ vertices in $B'$, this means they shatter any set of representatives of the sets $T^{(i)}_j$.)
\end{itemize}
The proof of Lemma~\ref{key} uses the so-called `sparsening method', together with a repeated application of the `reverse' Sauer's Lemma; that is, Sauer's Lemma combined with Lemma~\ref{backwards}, the observation that if $U \to V$, then $V \to U'$ for some (large) $U' \subset U$.

Now, suppose such a set $B'$ exists in $G$. We show (see Lemma~\ref{countA}) that in almost every such graph we can find subsets $W^{(i)}_j \subset T^{(i)}_j$ such that, for each $p,q \in [t]$, the pair $(W^{(1)}_p,W^{(2)}_q)$ is grey. Hence, by the Embedding Lemma, we can find a copy of $U(3,k)$ in $G$, for arbitrarily large $k$ (see Lemma~\ref{countB}). But this is a contradiction, since $\chi_c(\P) = 2$ (see Lemma~\ref{Urkfree}).

We have shown that $|B|$ is bounded for almost every $G \in \P$. Since $B$ is maximal, it follows that each vertex $v \in V(G)$ is a `clone' of some vertex $b \in B$ with respect to one of the sets $S_j$, i.e.,
$$|(\Gamma(v) \cap S_j) \triangle (\Gamma(b) \cap S_j)| \; \le \; \alpha n.$$
Since we expect to have few choices inside the sets $S_j$, it would be natural to expect that $v \in S_j$. Although this is not necessarily true for every vertex $v \in V(G)$, it turns out that, for almost every graph $G \in \P$, we can make it true by `adjusting' the partition $(S_1,S_2)$ (see Lemmas~\ref{countC} and~\ref{countD}). We obtain a new partition, $(S'_1,S'_2)$, which is `close' to the original partition (in the sense that $|S_j \triangle S'_j| \le \alpha n$), such that for each $j = 1,2$, and every $v \in S_j'$, $v$ is a clone of some $b \in B$ with respect to $S_j'$.

Finally, let $U_j \subset S'_j$ be the vertex set of a maximal collection of disjoint copies of $U(k)$ in $S_j$. (When $\chi_c(\P) \ge 3$ this step is more complicated, see the algorithm before Lemma~\ref{countU}). We claim that $|U_j| = O\left( n^{1-\eps} \right)$ for almost every $G \in \P$; to prove this, we simply count (see Lemma~\ref{countU}). First, note that there are at most $n^n$ choices for the partition of $V(G)$, and for the edges incident with $B$, and at most
$$2^{\alpha |U_j| n + n^{2-\eps}}$$
choices for the edges inside $S_j$. (This follows because each vertex of $U_j$ is a clone of a vertex in $B$ with respect to $S_j$, and by Theorem~\ref{count}, using the fact that $S'_j \setminus U_j$ is $U(k)$-free.) We will show further (see Lemma~\ref{noUk}) that we have at most $2^{(1/2 - 2\alpha)|U_j|n}$ choices for the edges between $U_j$ and $V(G) \setminus S'_j$. Thus the total number of choices, $|\P_n|$, satisfies
\begin{eqnarray*}
\log_2 \big( |\P_n| \big) & \le & \big( |S'_1| - |U_1| \big)\big( |S_2'| - |U_2|\big) \,+\, \left( \frac{1}{2} - \alpha \right) \big( |U_1| + |U_2| \big) n \, + \, O\left( n^{2-\eps} \right)\\
& \le & \frac{n^2}{4} \, - \, \alpha \big( |U_1| + |U_2| \big) n \, + \, O\left( n^{2-\eps} \right),
\end{eqnarray*}
which implies that $|U_1| + |U_2| = O\left( n^{1 - \eps} \right)$, as required. (We have assumed for simplicity that $|U_j| = o(n)$; the calculation in the other case is essentially the same.) Letting $A = U_1 \cup U_2 \cup B$, we obtain Theorem~\ref{structure}.

\section{Proof of Theorem~\ref{count}}\label{T2sec}

In this section we give a short proof of Theorem~\ref{count}. Our main tool is Sauer's Lemma.

\begin{proof}[Proof of Theorem~\ref{count}]
Let $k \in \N$, and let $G$ be a $U(k)$-free graph on $n$ vertices. We first claim that, given any bipartition $(A,B)$ of the vertex set $V(G)$, there are at most $2^{n^{2-\eps}}$ choices for the cross-edges.

Indeed, let $0 < \eps < 1/(k+1)$, and partition $B$ into sets $B_1, \ldots, B_t$ of size about $n^\eps$. By Sauer's Lemma, for each $j \in [t]$ we have
$$|\{ S \subset B_j \, : \, \exists \, a \in A \textup{ such that } \Gamma(a) \cap B_j = S \}| \; \le \; k {{|B_j|} \choose {k-1}} \; < \; n^{\eps k},$$
since $G[A,B_j]$ is $U(k)$-free. Thus the number of choices for $G[A,B]$ is at most
$$\prod_{j = 1}^t {{2^{|B_j|}} \choose n^{\eps k} } \left( n^{\eps k} \right)^n  \; \le \; 2^{\sum_j |B_j| n^{\eps k}} \left( 2^{\eps k n \log n} \right)^{n^{1 - \eps}} \; \le \; 2^{n^{2-\eps} \log n},$$
since $\sum_j |B_j| = n$ and $1 + \eps k < 2 - \eps$.

Finally, we may cover $E(G)$ with $\log n$ bipartite graphs, and so the number of choices for $G$ is at most $2^{n^{2-\eps} (\log n)^2} \ll 2^{n^{2-\eps'}}$ for any $\eps' < \eps$, as required.
\end{proof}

The reader will have noticed that the value of $\eps$ obtained above is not best possible; a more precise calculation is undertaken in Section~\ref{bipsec}.

\section{Some lemmas on universal graphs}\label{lemmasec}

In this section we state some of the lemmas we'll use to prove Theorem~\ref{structure}. We begin with a simple but key observation. Recall that we write $A \to B$ to mean that $A$ shatters $B$, i.e., $G[A',B] = U(k)$ for some $A' \subset A$, where $|B| = k$.

\begin{lemma}\label{backwards}
Let $G$ be a graph and let $t \in \N$. Suppose $A,B \subset V(G)$, with $A \to B$ and $|B| \ge 2^t$. Then there exist subsets $A' \subset A$ and $B' \subset B$ such that $B' \to A'$, and $|A'| = t$.
\end{lemma}

\begin{proof}
Let $A \to B$ be as described, and assume (taking subsets if necessary) that $|B| = 2^t$ and $G[A,B] = U(2^t)$. Label the elements of $B$ with the vertices of the hypercube $[2]^t$ arbitrarily, and consider the faces of this cube (i.e., the subcubes of dimension $t-1$) which contain the origin. Denote by $F_1, \ldots, F_t$ the corresponding subsets of $B$.

Now, there is an obvious bijection $\phi$ between vertices of $A$ and subsets of $B$ (a vertex is mapped to its neighbourhood), and so each set $F_j \subset B$ corresponds to a vertex of $A$. Let $A' = \{\phi(F_1), \ldots, \phi(F_t)\}$.

We claim that for each pair of vertices $b,b' \in B$, we have $\Gamma(b) \cap A' \neq \Gamma(b') \cap A'$. Indeed, if $b \neq b'$ then their labels must differ in some direction on the cube, so $b \in F_j$, $b' \notin F_j$, say. But then $\phi(F_j) \in \Gamma(b) \setminus \Gamma(b')$, as claimed. Thus $A'$ and $B'$ are sets as required by the lemma.
\end{proof}

We shall also use the following slight extension of Lemma~\ref{backwards}.

\begin{lemma}\label{back2}
Let $G$ be a graph and let $r,t \in \N$. Let $A_1, \ldots, A_r, B \subset V(G)$ be disjoint sets, with $|B| \ge 2^{rt}$, and $A_j \to B$ for each $j \in [r]$. Then there exist subsets $A_j' \subset A_j$ and $B' \subset B$ such that $B' \to \bigcup_j A_j'$, and $|A_j'| = t$ for each $j \in [r]$.
\end{lemma}

\begin{proof}
Assume (by taking a subset if necessary) that $|B| = 2^{rt}$. By Lemma~\ref{backwards}, there exist subsets $A^*_j \subset A_j$ such that $|A^*_j| = rt$ and $B \to A^*_j$ for each $j \in [r]$. Moreover, we can choose the sets $A^*_j$ so that the following holds:
\begin{itemize}
\item Let $(v_1^{(1)},\ldots,v_1^{(rt)})$ be an arbitrary order for the elements of $A^*_1$. Then, for each $i \in [rt]$ and $j \in [r]$, there exists a vertex $v^{(i)}_j \in A^*_j$ with the same neighbourhood in $B$ as $v^{(i)}_1$.
\end{itemize}
Indeed, to do this we simply use the same $[2]^t$-labelling of $B$ (see the proof of Lemma~\ref{backwards} above) for each set $A_j$.

Now we simply have to choose $r$ disjoint $t$-subsets $A'_j \subset A^*_j$ for $j \in [r]$. To spell it out, let $A'_j = \{v^{(i)}_j : (j-1)t + 1 \le i \le jt\}$ for each $j \in [r]$. It is clear that $B \to \bigcup_j A'_j$, as required.
\end{proof}

In Section~\ref{sketchsec} we used the $3$-partite graph $U(3,k)$. We now make the natural generalization to $r$-partite graphs, which we shall denote $U(r,k)$.

\begin{defn}
For each $k,r \in \N$, define the \emph{generalized universal graph}, $U(r,k)$, to be the $r$-partite graph on vertex set $A_1 \cup \ldots \cup A_r$, where $|A_1| = k$ and $|A_{j+1}| = 2^{\sum_1^j |A_i|}$ for each $1 \le j \le r-1$, such that
$$A_{j+1} \, \to \, A_1 \cup \ldots \cup A_j$$
for each $1 \le j \le r - 1$.

Moreover, for each $v \in \{0,1\}^r$, define $U^*_v(r,k)$ to be the graph on vertex set $A_1 \cup \ldots \cup A_r$ such that the induced $r$-partite graph is $U(r,k)$, and $G[A_j]$ is either complete or empty for each $j \in [r]$, with $G[A_j]$ complete if and only if $v_j = 1$.
\end{defn}

We next apply Lemma~\ref{backwards} to prove a Ramsey-type theorem for the graphs $U(r,k)$.

\begin{lemma}\label{UrkRamsey}
For each $r,k \in \N$, there exists $K = K(r,k) \in \N$ such that the following holds. Let $G$ be a graph on vertex set $A_1 \cup \ldots \cup A_r$, and suppose $G[A_1,\ldots,A_r] = U(r,K)$. Then $U^*_v(r,k) \le G$ for some $v \in \{0,1\}^r$.
\end{lemma}

\begin{proof}
The proof is by induction on $r$. For $r = 1$ the lemma is just Ramsey's Theorem. For $r = 2$ we first apply Ramsey to $A_1$, to obtain a subset $B_1 \subset A_1$ such that $G[B_1]$ is either complete or empty. Note that $A_2 \to B_1$, so by Lemma~\ref{backwards}, there exists a set $B'_2 \subset A_2$ such that $B_1 \to B'_2$. Finally, applying Ramsey to $B'_2$, we obtain a subset $B_2 \subset B'_2$ such that $G[B_2]$ is complete or empty, and $B_1 \to B_2$. 

So let $r \ge 3$, and assume the result holds for smaller values of $r$. Let $t = |U^*_{v'}(r-1,k)|$, $T = 2^t$ and $m = {{R(T)} \choose T}$, where $R(T)$ is the Ramsey number. Let $K' = m k$, and let $K = K(r-1,K')$. We claim that the lemma holds for $K$.

First, by the induction hypothesis, there exists a copy $H$ of $U^*_{v'}(r-1,K')$ in $G[A_1 \cup \ldots \cup A_{r-1}]$, for some $v' \in \{0,1\}^{r-1}$. Note that $A_r \to V(H)$, and let $V(H) = B_1 \cup \ldots \cup B_{r-1}$, where $B_{j+1} \to B_1 \cup \ldots \cup B_j$ for each $1 \le j \le r - 2$.

Since $A_r \to V(H)$, the bipartite graph $G[A_r,V(H)]$ contains every (small) bipartite graph as an induced subgraph. We shall define a specific such bipartite graph, $F$, and show that it contains $U^*_v(r,k)$ for some $v \in \{0,1\}^r$.

Indeed, since $K' = m k$, $H$ contains $m$ disjoint copies of $U^*_{v'}(r-1,k)$. To see this, simply partition $B_1$ into $m$ equal-size parts, $B^{(1)}_1, \ldots, B^{(m)}_1$, and successively choose disjoint sets $B^{(1)}_j, \ldots, B^{(m)}_j \subset B_j$ such that $B^{(i)}_j \to B^{(i)}_1 \cup \ldots \cup B^{(i)}_{j-1}$. Let $H^{(i)}$ denote the graph induced by $B_1^{(i)} \cup \ldots \cup B^{(i)}_{r-1}$.

Now, define the bipartite graph $F$ as follows. Let $\I(T)$ denote the set of subsets of $\{1,\ldots,R(T)\}$ of size $T$.
\begin{itemize}
\item[$(a)$] $V(F) = X \cup Y$, with $X = [R(T)]$ and $Y = \ds\bigcup_{Z \in \I(T)} Y_Z$, where the sets $Y_Z$ are disjoint, and $|Y_Z| = t$ for each $Z \in \I(T)$.
\item[$(b)$] For each subset $W \subset X$ of size $T$, we have $W \to Y_W$.
\item[$(c)$] The other edges may be chosen arbitrarily.
\end{itemize}
Note that no two vertices of $X$ have the same neighbourhood in $Y$, since each such pair is contained in some $T$-set $W \subset X$, so differs on $Y_W$. Therefore, since $A_r \to V(H)$, it follows that there exists a set $A_r' \subset A_r$ such that
$$G\left[ A_r', V\big( H^{(1)} \big) \cup \ldots \cup V \big( H^{(m)} \big) \right] \; = \; F,$$
with $Y_Z = V(H^{(\phi(Z))})$ for each $Z \in \I(T)$, for some bijection $\phi: \I(T) \to [m]$.

Finally we apply Ramsey's Theorem to $A_r'$, to obtain a set $B_r$ of size $T$ such that $G[B_r]$ is complete or empty. By the definition of $F$ we have $B_r \to Y_{B_r}$, so $B_r \to V(H^{(i)})$ for some $i \in [m]$. It follows that the set $B_r \cup Y_{B_r}$ induces a copy of $U^*_v(r,k)$ for some $v \in \{0,1\}^r$, as required.
\end{proof}

The following immediate consequence of Lemma~\ref{UrkRamsey} says that if $\chi_c(\P) = r$, and there is an arbitrarily large copy of $U(r+1,k)$ in $G$, then $G$ contains a forbidden graph of $\P$.

\begin{lemma}\label{Urkfree}
Let $r,m \in \N$, and for each $v \in \{0,1\}^{r}$ choose a graph $H_v \in \HH(r,v)$ with $|V(H_v)| \le m$. Let $k = k(r,m) \in \N$ be sufficiently large, and let $G$ be any graph.

Then either $G$ is $U(r,k)$-free, or $H_v \le G$ for some $v \in \{0,1\}^r$.
\end{lemma}

\begin{proof}
Suppose that $G$ contains a copy of $U(r,k)$. Then, by Lemma~\ref{UrkRamsey}, $G$ also contains an induced copy of $U^*_v(r,m)$ for some $v \in \{0,1\}^r$. But $H_v \le U^*_v(r,k)$ for every $H_v \in \HH(r,v)$ with $|H_v| \le m$, and so we are done.
\end{proof}

Before continuing, we note that we have already proved Theorem~\ref{structure} in the case $r = 1$.

\begin{cor}
Let $\P$ be a hereditary property of graphs with $\chi_c(\P) = 1$. Then there exists $k = k(\P) \in \N$ such that $G$ is $U(k)$-free for every $G \in \P$, and hence there exists $\eps > 0$ such that $$|\P_n| \; \le \; 2^{n^{2-\eps}}$$
for every sufficiently large $n \in \N$.
\end{cor}

\begin{proof}
Since $\chi_c(\P) = 1$, there exists, for each $v \in \{0,1\}^2$, a graph $H_v \in \HH(2,v)$ such that $H_v \not\in \P$. By Lemma~\ref{Urkfree}, it follows that there exists $k \in \N$ such that $G$ is $U(k)$-free for every $G \in \P$. By Theorem~\ref{count}, it follows that
$$|\P_n| \; \le \; 2^{n^{2-\eps}}$$
for every sufficiently large $n \in \N$, as required.
\end{proof}

The following `induced' embedding lemma is a simple consequence of the Embedding Lemma, the Slicing Lemma and (the proof of) Lemma~\ref{Urkfree}.

\begin{lemma}\label{embed2}
Given $\delta > 0$ and $m,r \in \N$, there exist $\eps > 0$ and $n_0 = n_0(\eps,\delta,m,r) \in \N$ such that the following holds. Let $G$ be a graph on $A_1 \cup \ldots \cup A_r$, where $|A_j| \ge n_0$ for each $j \in [r]$, and suppose each pair $(A_i,A_j)$ is $(\eps,\delta)$-grey.

Then for some $v \in \{0,1\}^r$, $H \le G$ for every $H \in \HH(r,v)$ with $|H| \le m$.
\end{lemma}

\begin{proof}
We claim that $G$ is not $U(r,k)$-free, where $k = k(r,m)$ is the constant in Lemma~\ref{Urkfree}, if $n_0$ is sufficiently large. Indeed, by the Slicing Lemma, if we partition each set $A_j$ into $t = |U(r,k)|$ almost equal sets $(A_j^{(1)},\ldots,A_j^{(t)})$, then each pair $(A^{(i)}_j,A^{(i')}_{j'})$ (with $j \neq j'$) is $(\eps',\delta')$-grey, where $\eps' = 2\eps/t$ and $\delta'  = \delta/t - \eps$. Therefore, by the Embedding Lemma, there exists a copy of $U(r,k)$ in $G$, as required.

But now, by the proof of Lemma~\ref{Urkfree}, $G$ also contains an induced copy of $U^*_v(r,m)$ for some $v \in \{0,1\}^r$, and so $H_v \le U^*_v(r,k)$ for every $H_v \in \HH(r,v)$ with $|H_v| \le m$, as required.
\end{proof}

We next prove our key lemma whose proof uses the so-called `sparsening' method.

\begin{lemma}\label{key}
For each $\alpha > 0$ and $t \in \N$, there exist $c = c_1(\alpha,t) \in \N$ and $\delta = \delta_1(c,\alpha,t) > 0$ such that the following holds. Let $G$ be a bipartite graph with parts $U$ and $V$, satisfying $|U| \ge c$, $|V| = n \in \N$ and
$$\big| \Gamma(u) \, \triangle \, \Gamma(u') \big| \; \ge \; \alpha n$$
for each $u,u' \in U$ with $u \neq u'$.

Then there exists a subset $U' \subset U$, with $|U'| = t$, and sets $T_1, \ldots, T_{2^t} \subset V$, with $|T_j| \ge \delta n$ for each $j \in [2^t]$, such that the following holds:
\begin{itemize}
\item[$(a)$] If $u,v \in T_j$ then $\Gamma(u) \cap U' = \Gamma(v) \cap U'$.
\item[$(b)$] If $W = \{w_1,\ldots,w_{2^t}\}$, where $w_j \in T_j$ for each $j \in [2^t]$, then $W \to U'$.
\end{itemize}
\end{lemma}

\begin{proof}
Assume that $|U| = c \ge \left( \frac{5 \log c}{\alpha} \right)^{2^t}$, and with foresight, let $p = \ds\frac{5 \log c}{\alpha n}$. We claim that there exists a subset $X \subset V$, with $|X| = pn$, such that
$$\Gamma(x) \cap X \; \neq \; \Gamma(y) \cap X$$
for each $x,y \in U$ with $x \neq y$. Indeed, if we choose a random subset $X \subset V$ of size $pn$, then
$$\Pr\Big( \Gamma(x) \cap X = \Gamma(y) \cap X \Big) \; \le \; {{n - \alpha n} \choose {pn}} \Big/ {n \choose {pn}} \; \le \; (1 - \alpha)^{p n}$$
for each such pair $\{x,y\}$, and so
$$\Ex \left| \left\{ \{x,y\} \in {U \choose 2} \,:\, \Gamma(x) \cap X = \Gamma(y) \cap X \right\} \right| \; \le \; {{|U|} \choose 2} (1-\alpha)^{pn} \; \le \; c^2 e^{-p\alpha n} \; < \; 1,$$
by our choice of $p$.

Thus such a set $X$ must exist, as claimed. Now, since $c$ is sufficiently large so that
$$c \; \ge \; |X|^{2^t} \; = \; (pn)^{2^t} \; = \; \left( \frac{5 \log c}{\alpha} \right)^{2^t}$$
then, by Sauer's Lemma, there exist sets $U^* \subset U$ and $X^* \subset X$, with $|X^*| \ge 2^t$, such that $U^* \to X^*$. Thus, by Lemma~\ref{backwards}, there exist sets $U_1 \subset U^*$ and $X_1 \subset X^*$, with $|U_1| = t$, such that $X_1 \to U_1$.

Now, let us remove $X_1$ from $V$, and repeat the process, obtaining disjoint sets $X_2, \ldots, X_\ell$. Since $|X_j| = 2^t$, we can do this so long as $\ell \le \ds\frac{\alpha n}{2^{t+1}}$. (It is easy to see that nothing goes wrong in the calculation above when we replace $\alpha$ by $\alpha/2$.) By the pigeonhole principle, there is a set $U' \subset U$ which occurs (as the set $U_j$) at least
$$\ell {{|U|} \choose t}^{-1} \; \ge \; \frac{\alpha n}{2^{t+1} c^t} \; \ge \; \delta n$$
times. Let $\J = \{j : X_j \to U'\}$, and note that $|\J| \ge \delta n$. Write $X_j = \{x_j(1),\ldots,x_j(2^t)\}$, where $\Gamma(x_j(k)) \cap U' = \Gamma(x_{j'}(k)) \cap U'$ for every $j,j' \in \J$, and let $T_k = \{x_j(k) : j \in \J\}$, for each $k \in [2^t]$. The sets $T_1, \ldots, T_{2^t}$ satisfy conditions $(a)$ and $(b)$, as required.
\end{proof}

The corresponding result for $r$-partite graphs follows as an easy corollary.

\begin{lemma}\label{key2}
For each $\alpha > 0$ and $r,t \in \N$, there exist $c = c_2(\alpha,r,t) \in \N$ and $\delta = \delta_2(c,\alpha,r,t) > 0$ such that the following holds. Let $G$ be a graph on $n \in \N$ vertices, let $(S_1,\ldots,S_r)$ be a partition of $V(G)$, and let $B \subset V(G)$ satisfy $|B| \ge c$ and
$$\left| \big( \Gamma(b) \cap S_j \big) \triangle \big( \Gamma(b') \cap S_j \big) \right| \; \ge \; \alpha n$$
for every $j \in [r]$, and each $b,b' \in B$ with $b \neq b'$.

Then there exists a subset $B' \subset B$, with $|B'| = t$, and sets $T^{(i)}_1, \ldots, T^{(i)}_{2^t} \subset S_i$ for each $i \in [r]$, with $|T^{(i)}_j| \ge \delta n$, such that the following holds:
\begin{itemize}
\item[$(a)$] If $u,v \in T^{(i)}_j$ then $\Gamma(u) \cap B' = \Gamma(v) \cap B'$.
\item[$(b)$] If $W = \{w_1,\ldots,w_{2^t}\}$, where $w_j \in T^{(1)}_j \cup \ldots \cup T^{(m)}_j$ for each $j \in [2^t]$, then $W \to B'$.
\end{itemize}
\end{lemma}

\begin{proof}
Let $t_1 > t_2 > \dots > t_{r+1} = t$ be a sequence satisfying $t_j \ge c_1(\alpha,t_{j+1})$ for each $j \in [r]$ (where $c_1$ is the function in Lemma~\ref{key}), and assume that $|B| = c \ge t_1$. Applying Lemma~\ref{key} (with $t = t_1$) to the pair $(B,S_1) = (U,V)$, we obtain a set $B_1 \subset B$ and a collection of sets $T^{(1)}_1, \ldots, T^{(1)}_{2^{t_1}} \subset S_1$ given by that lemma. In particular, we have $|B_1| \ge t_2$ and $|T^{(i)}_j| \ge \delta n$, where $\delta = \delta_1(t_1,\alpha,t_2)$.

Similarly, for each $q \in [2,r]$ we may apply Lemma~\ref{key} (with $c = t_q$ and $t = t_{q+1}$) to the pair $(B_{q-1},S_q)$, to obtain sets $B_q \subset B_{q-1}$ and $T^{(q)}_1, \ldots, T^{(q)}_{2^{t_q}} \subset S_{q}$, with $|B_q| = t_{q+1}$ and $|T^{(q)}_j| \ge \delta n$, where $\delta = \delta_1(t_q,\alpha,t_{q+1})$. Let $\delta_2(c,\alpha,t) = \min_q \{ \delta_1(t_q,\alpha,t_{q+1})\}$.

Finally, for each $q \in [r]$, re-number so that the sets $T^{(q)}_1, \ldots, T^{(q)}_{2^t}$ shatter $B_r \subset B_q$. It follows that the sets $B' = B_r$ and $\{ T^{(i)}_j : i \in [r], j \in [2^t]\}$ are those required by the lemma, and so we are done.
\end{proof}

We shall in fact use the following immediate corollary of Lemmas~\ref{back2} and~\ref{key2}.

\begin{lemma}\label{key3}
For each $\alpha > 0$ and $r,t \in \N$, there exist $c = c_3(\alpha,r,t) \in \N$ and $\delta = \delta_3(c,\alpha,r,t) > 0$ such that the following holds. Let $G$ be a graph on $n \in \N$ vertices, let $(S_1,\ldots,S_r)$ be a partition of $V(G)$, and let $B \subset V(G)$ satisfy $|B| \ge c$ and
$$\left| \big( \Gamma(b) \cap S_j \big) \triangle \big( \Gamma(b') \cap S_j \big) \right| \; \ge \; \alpha n$$
for every $j \in [r]$, and each $b,b' \in B$ with $b \neq b'$.

Then there exists a subset $B' \subset B$, with $|B'| = 2^{rt}$, and sets $T^{(i)}_1, \ldots, T^{(i)}_t \subset S_i$ for each $i \in [r]$, with $|T^{(i)}_j| \ge \delta n$, such that the following holds:
\begin{itemize}
\item[$(a)$] If $u,v \in T^{(i)}_j$ then $\Gamma(u) \cap B' = \Gamma(v) \cap B'$.
\item[$(b)$] If $W = \{w_{11},\ldots,w_{rt}\}$, where $w_{ij} \in T^{(i)}_j$ for each $i \in [r]$, $j \in [t]$, then $B' \to W$.
\end{itemize}
\end{lemma}

\begin{proof}
First we apply Lemma~\ref{key2} to get sets $U^{(i)}_j$ (for each $i \in [r]$ and $j \in [2^{|B'|}]$) which shatter $B'$. Applying Lemma~\ref{back2} to these sets (or, if the reader prefers, to an arbitrarily chosen element from each set) gives the required sets.
\end{proof}

The following two observations will be useful in Section~\ref{proofsec}.

\begin{obs}\label{nongrey}
Let $\delta > 0$ be sufficiently small, and let $|A| = |B| = n$. There are at most $2^{\delta n^2}$ bipartite graphs on $A \cup B$ of density at most $\delta^2$.
\end{obs}

\begin{proof}
If $G$ is such a graph then $e(G) \le m = \delta^2 n^2$, so the number of choices is at most
$$\sum_{j=0}^m {{n^2} \choose j} \; \le \; 2{{n^2} \choose m} \; \le \; 3^m \left( \frac{1}{\delta^2} \right)^{\delta^2 n^2} \; < \; 2^{\delta n^2},$$
as required, since $\left( 1/x^2 \right)^x \to 1$ as $x \to 0$.
\end{proof}

Recall that $K_r(t)$ denotes the Tur\'an graph on $rt$ vertices, i.e., the complete $r$-partite graph with $t$ vertices in each part.

\begin{obs}\label{findKtt}
For each  $r,t \in \N$, there exist $\eps > 0$ and $n_0 = n_0(r,t) \in \N$ such that the following holds. Let $G$ be an $r$-partite graph on vertex set $A_1 \cup \ldots \cup A_r$ and, for each $j \in [r]$, let $B_j(1) \cup \ldots \cup B_j(t)$ be an equipartition of $A_j$. Suppose $n \ge n_0$, $|A_1| = \ldots = |A_r| = n$, and $e(G) \ge (1 - \eps){r \choose 2}n^2$.

Then there exists a copy $H$ of $K_r(t)$ in $G$ with $|H \cap B_j(k)| \le 1$ for each $j,k$.
\end{obs}

\begin{proof}
Since $\eps$ may be chosen so that $\eps r^3 t^3 \ll 1$, the result is trivial by the greedy algorithm. To spell it out, for each $i,j$ there exists a vertex $v \in B^{(i)}_j$ such that $\Gamma(v)$ avoids at most $|B^{(i')}_{j'}|/rt$ vertices of each other set $B^{(i')}_{j'}$.
\end{proof}

Finally, we prove the following easy lemma, which bounds the number of ways in which a copy of $U(r,k)$ in $G$ can be attached to the rest of the graph without creating a copy of $U(r+1,k)$.

\begin{lemma}\label{noUk}
For each $r,k \in \N$, there exists $K = K(r,k) \in \N$ and $\delta = \delta(r,k) > 0$ such that the following holds. Given a vertex set $A \cup B$, with $|A| = |U(r,k)|$ and $|B| = n$, let
$$\G(r,k,n) \; := \; \Big\{ G[A,B] \,:\, \exists \textup{ a $U(r+1,k)$-free graph $G$ on $A \cup B$ with $G[A] = U(r,k)$} \Big\},$$
the set of bipartite graphs on $A \cup B$ which do not create a copy of $U(r+1,k)$. Then
$$|\G(r,k,n)| \; \le \; 2^{|U(r,k)|n - \delta n}.$$
\end{lemma}

\begin{proof}
Since $G$ is $U(r+1,k)$-free, no set $X \subset B$ shatters $A$. This mean that there exists a subset $Y \subset A$ such that $\Gamma(u) \cap A \neq Y$ for every $u \in B$. We therefore have at most
$$\left( 2^{|A|} - 1 \right)^n \; \le \; 2^{|U(r,k)|n - \delta n}$$
choices for the edges of $G[A,B]$, where $\delta = 2^{-|U(r,k)|}$, as required.
\end{proof}

\section{Proof of Theorem~\ref{structure}}\label{proofsec}

In this section we shall describe several `bad' properties of a graph, and prove that the number of graphs in a hereditary property of graphs $\P$ with one of these properties is $o(|\P_n|)$. We then deduce Theorem~\ref{structure} by observing that all remaining graphs have the required structure.

We begin with an important definition, motivated by~\cite{BBS1}.

\begin{defn}
Let $\eps, \delta, \gamma > 0$, let $r \in \N$, let $G$ be a graph, and let $V(G) = S_1 \cup \ldots \cup S_r$ be a partition of $V(G)$. We say that $P = (S_1, \ldots, S_r)$ is a \emph{BBS-partition} of $G$ (for $(\eps,\delta,\gamma)$) if there exists a Szemer\'edi partition of $G$ (for $\eps$) into $m$ parts (for some $1/\eps < m \in \N$) such that:
\begin{itemize}
\item[$(a)$] Each part $S_j$ is a union of (an almost equal number of) Szemer\'edi sets.
\item[$(b)$] Each part $S_j$ contains at most $\gamma m^2$ pairs which are $(\eps,\delta)$-grey.
\end{itemize}
\end{defn}

Note that if $(S_1,\ldots,S_r)$ is a BBS-partition of $G$ for $(\eps,\delta,\gamma)$, then it follows that
$$\left(\frac{1}{r} - \eps \right) n \; \le \; |S_j| \; \le \; \left( \frac{1}{r} + \eps \right) n$$
for each $j \in [r]$, by condition $(a)$.

The following lemma was proved (in the monotone case) by Balogh, Bollob\'as and Simonovits~\cite{BBS1} using the Szemer\'edi Regularity Lemma, the Erd\H{o}s-Simonovits Stability Theorem, and the Embedding Lemma. The proof is essentially the same in our case, but for the sake of completeness we shall give a fairly complete sketch.

\begin{lemma}[Balogh, Bollob\'as and Simonovits~\cite{BBS1}]\label{S1S2}
Let $r \in \N$, let $\gamma > 0$, and let $\delta = \delta(\gamma,r) > 0$ and $\eps = \eps(\delta,\gamma,r) > 0$ be sufficiently small. Let $\P$ be a hereditary property of graphs with $\chi_c(\P) = r$.

For almost every graph $G \in \P$, there exists a BBS-partition of $G$ for $(\eps,\delta,\gamma)$ into $r$ parts.
\end{lemma}

\begin{proof}
First, since $\chi_c(\P) = r$, it follows that for each $v \in \{0,1\}^{r+1}$ there exists a `forbidden' graph $H_v \in \HH(r+1,v)$, such that $H_v \not\in \P$. Choose such a graph for each $v \in \{0,1\}^{r+1}$, and let $t = \max\{|V(H_v)| : v \in \{0,1\}^{r+1}\}$.

Let $G \in \P_n$, with $n \in \N$ sufficiently large, and apply the Szemer\'edi Regularity Lemma (for $\eps$ and $m = 1/\eps$) to the graph $G$. We obtain a collection of $k$ parts, $B_1, \ldots, B_k$; define a graph $H$ on $k$ vertices by letting $ij \in E(H)$ if and only if the pair $(B_i,B_j)$ is $(\eps,\delta^2)$-grey. (This is called the cluster graph of $G$.)

Suppose first that $K_{r+1} \subset H$. Then, applying Lemma~\ref{embed2}, we deduce that $H \le G$ for every $H \in \HH(r+1,v)$ with $|H| \le t$, for some $v \in \{0,1\}^{r+1}$. Therefore $H_v \in \P$, which is a contradiction.

Suppose next that the number of edges in the cluster graph $H$ satisfies
$$e(H) \; \le \; \left( 1 - \frac{1}{r} - 2\delta \right) {k \choose 2}.$$
It is easy to bound the number of graphs $G$ with at most this many edges. Indeed, there are at most $n^n$ ways of choosing the Szemer\'edi partition, and, by Observation~\ref{nongrey} and our choice of $m$, at most $2^{(\eps + \delta)n^2}$ ways of choosing the edges inside the parts, and between non-grey pairs. Moreover, there are at most $2^{(1 - 1/r - 2\delta)n^2/2}$ ways of choosing the edges between grey pairs. But
$$|\P_n| \; \ge \; 2^{(1 - 1/r + o(1))n^2/2},$$
by Observation~\ref{ABT}, so the number of such graphs $G$ on $n$ vertices is $o(|\P_n|)$.

Hence we may assume that $H$ is $K_{r+1}$-free, and has at least $(1 - 1/r - 2\delta){k \choose 2}$ edges. By the Erd\H{o}s-Simonovits Stability Theorem, it follows that we can change $H$ into the Tur\'an graph $T_r(k)$ by changing at most $\gamma k^2$ edges. But this is exactly the definition of a BBS-partition, and so we are done.
\end{proof}

Next, we need to count those graphs which have large `irregularities' between pairs $(S_i,S_j)$ of their BBS-partition. The following definition is designed to allow us to take advantage of the $\delta n$-sets given by Lemma~\ref{key3}.

Let $\eps, \delta, \gamma, \alpha > 0$, let $n \in \N$, and let $\P$ be a hereditary property of graphs. To simplify the notation in what follows, we shall suppress dependence on $\eps$, $\delta$ and $\gamma$. Define a set $\A(\P_n,\alpha) \subset \P_n$ as follows:
\begin{align*}
& \A(\P_n,\alpha) \; := \; \big\{ G \in \P_n \,:\, \exists\textup{ a BBS-partition $(S_1,\ldots,S_r)$ of $G$ for $(\eps,\delta,\gamma)$, and sets $X \subset S_i$} \\
& \hspace{1.5cm} \textup{ and $Y \subset S_j$, for some $i \neq j$, with $|X|, |Y| \ge \alpha n$, such that } d(X,Y) \not\in (\delta,1-\delta) \big\}.
\end{align*}
The following lemma says that the collection $\A(\P_n,\alpha)$ is small.

\begin{lemma}\label{countA}
Let $\alpha > 0$, let $2 \le r \in \N$, and let $\eps > 0$, $\delta > 0$ and $\gamma > 0$ be sufficiently small. Let $\P$ be a hereditary property of graphs with $\chi_c(\P) = r$, and let $n \in \N$. Then
$$|\A(\P_n,\alpha)| \; \le \; 2^{(1 - 1/r) n^2/2 \, - \, \alpha^2 n^2/3}  \; = \; o( |\P_n| ).$$
\end{lemma}

\begin{proof}
For the first inequality we simply count. We have at most $n^n$ choices for the BBS-partition $S_1 \cup \ldots \cup S_r$, and the sets $X \subset S_i$ and $Y \subset S_j$. By Observation~\ref{nongrey}, and the definition of a BBS-partition, we have at most
$$2^{(\eps + \sqrt{\delta} + \gamma)n^2}$$ choices for the edges inside the set $S_k$, for each $k \in [r]$.

Next, recall that $(1/r - \eps)n \le |S_\ell| \le (1/r + \eps)n$ for each $\ell \in [r]$, by the definition of a BBS-partition, and so we have at most
$$2^{(1/r + \eps)^2 n^2}$$
choices for the edges between $S_p$ and $S_q$, for each $p \neq q$. Moreover, we have at most
$$2^{(1/r + \eps)^2 n^2 - (1 - \sqrt{\delta})\alpha^2 n^2 + 1}$$
choices for the edges between $S_i$ and $S_j$. To see this, assume for simplicity that $|X| = |Y| = \alpha n$, and observe that we have at most $2^{(1/r + \eps - \alpha)^2 n^2}$ choices for the edges between $S_i \setminus X$ and $S_j \setminus Y$, at most $2^{(1/r + \eps - \alpha)\alpha n^2}$ choices for the edges between $X$ and $S_j \setminus Y$ (and similarly for those between $Y$ and $S_i \setminus X$), and, by Observation~\ref{nongrey}, at most $2^{\sqrt{\delta} \alpha^2 n^2 + 1}$ choices for the edges between $X$ and $Y$.

Putting these bounds together, we obtain
$$\log_2 \big( |\A(\P_n,\alpha)| \big) \; \le \; {r \choose 2}(1/r + \eps)^2 n^2 - \frac{\alpha^2 n^2}{2} + O\left( (\eps + \sqrt{\delta} + \gamma)n^2 \right).$$
The first inequality now follows if $\eps$, $\delta$ and $\gamma$ are sufficiently small. The final inequality follows by Observation~\ref{ABT}.
\end{proof}

Next, given a graph $G$, and a partition $P = (S_1,\ldots,S_r)$ of $G$, we say that a set of vertices $B \subset V(G)$ is $\alpha$-\emph{bad} for $(G,P)$ if
$$\left| \big( \Gamma(u) \cap S_j \big) \triangle \big( \Gamma(v) \cap S_j \big) \right| \; \ge \; \alpha n$$
for each $u,v \in B$ with $u \neq v$, and each $j \in [r]$. Let
$$B(G,P,\alpha) \; := \; \max\big\{ |B| \,:\, B \subset V(G) \textup{ is $\alpha$-bad for } (G,P)\big\}.$$
Now, given $\eps, \delta, \gamma, \alpha > 0$, $c,n \in \N$, and a hereditary property of graphs $\P$, we define a set $\B(\P_n,\alpha,c) \subset \P_n$ as follows:
\begin{align*}
& \B(\P_n,\alpha,c) \; := \; \big\{ G \in \P_n \,:\, \exists\textup{ a BBS-partition $P$ of $G$ for $(\eps,\delta,\gamma)$ with $B(G,P,\alpha) \ge c$} \big\}.
\end{align*}
We next show that, if $c = c(\P)$ is sufficiently large then the collection $\B(\P_n,\alpha,c)$ is small.

\begin{lemma}\label{countB}
Let $\alpha > 0$ and $r \in \N$, and let $\P$ be a hereditary property of graphs with $\chi_c(\P) = r$. There exist constants $c = c(\P,\alpha) \in \N$ and $\alpha' = \alpha'(\P,\alpha) > 0$ such that the following holds. Let $\eps > 0$, $\delta > 0$ and $\gamma > 0$ be sufficiently small, and let $n \in \N$ be sufficiently large. Then
$$\B(\P_n,\alpha,c) \; \subset \; \A(\P_n,\alpha').$$
\end{lemma}

\begin{proof}
First, choose a graph $H_v \in \HH(r+1,v) \setminus \P$ for each $v \in \{0,1\}^{r+1}$. Such graphs must exist because $\chi_c(\P) = r$. Let $m = \max\{ |H_v| : v \in \{0,1\}^{r+1}\}$, and let $k = k(r,m) \in \N$ be the constant in Lemma~\ref{Urkfree}.

Now, let $t = t(\alpha,r,k) \in \N$ be sufficiently large, let $c = c_3(\alpha,r,t)$ and $\tilde{\alpha} = \delta_3(c,\alpha,r,t)$ be the constants in Lemma~\ref{key3}, and let $\alpha' = \tilde{\alpha}/M(n_0,\eps)$, where $M(.,.)$ is the constant in the refinement version of Szemer\'edi's Lemma, and $n_0 = n_0(r,t)$ is the constant in Observation~\ref{findKtt}.

Let $G \in \P_n \setminus \A(\P_n,\alpha')$, let $\eps$, $\delta$ and $\gamma$ be sufficiently small, and suppose that there exists a BBS-partition $P$ of $G$ for $(\eps,\delta,\gamma)$ such that $B(G,P,\alpha) \ge c$, i.e., there exists a set $B \subset V(G)$, with $|B| \ge c$, which is $\alpha$-bad for $(G,P)$.

\medskip
\noindent \ul{Claim}: $G$ is not $U(r+1,k)$-free.

\begin{proof}[Proof of claim]
By Lemma~\ref{key3}, there exists a set $B' \subset B$, with $|B'| = 2^{rt}$, and disjoint sets $T^{(i)}_j \subset S_i$, with $|T^{(i)}_j| = \tilde{\alpha} n$ for each $i \in [r]$ and $j \in [t]$, such that
\begin{itemize}
\item[$(a)$] If $u,v \in T^{(i)}_j$ then $\Gamma(u) \cap B' = \Gamma(v) \cap B'$.
\item[$(b)$] If $W = \{w_{11},\ldots,w_{rt}\}$, where $w_{ij} \in T^{(i)}_j$ for each $i \in [r]$, $j \in [t]$, then $B' \to W$.
\end{itemize}
Let $T = \bigcup_{i,j} T^{(i)}_j$, and apply the refinement version of Szemer\'edi's Lemma (for $\eps$) to the (equi-)partition
$$\bigcup_{i = 1}^r \bigcup_{j = 1}^{t} T^{(i)}_j$$
of $T$. We obtain, for each $i \in [r]$ and $j \in [t]$, a partition $(U^{(i)}_j(1), \ldots, U^{(i)}_j(m))$ of $T^{(i)}_j$, such that the resulting partition of $T$ is a Szemer\'edi partition. Moreover, by our choices of constants above, we have $m \ge n_0(r,t)$ and $|U^{(i)}_j| \ge \alpha' n$ for every $i \in [r]$ and $j \in [t]$.

Suppose first that there exists a pair $(U^{(i)}_j(\ell), U^{(i')}_{j'}(\ell'))$, where $i \neq i'$, which is $\eps$-regular but not $(\eps,\delta)$-grey. Then the graph $G[U^{(i)}_j(\ell), U^{(i')}_{j'}(\ell')]$ has density in $[0,\delta) \cup (1-\delta,1]$, and so $G \in \A(\P_n,\alpha')$, a contradiction.

Thus every $\eps$-regular pair $(U^{(i)}_j(\ell), U^{(i')}_{j'}(\ell'))$ is also $(\eps,\delta)$-grey. By the definition of a Szemer\'edi partition, at most $\eps (mrt)^2$ pairs are irregular, and so the number of pairs $(U^{(i)}_j(\ell), U^{(i')}_{j'}(\ell'))$ with $i \neq i'$ which are not $(\eps,\delta)$-grey is at most $(1 - 2\eps){r \choose 2} (mt)^2$.

We apply Observation~\ref{findKtt} to the $r$-partite graph $F$ where $V(F) = \{U^{(i)}_j(\ell) : i \in [r], j \in [t], \ell \in [m]\}$, and a pair of vertices $\{U^{(i)}_j(\ell), U^{(i')}_{j'}(\ell')\}$ (with $i \neq i'$) is an edge of $F$ if and only if they form an $(\eps,\delta)$-grey pair. Since we chose $\eps > 0$ sufficiently small, and $m \ge n_0$, it follows that there exist representatives $\{ W^{(i)}_j : i \in [r], j \in [t]\}$, where $W^{(i)}_j \in \{U^{(i)}_j(1), \ldots, U^{(i)}_j(m)\}$, such that \emph{every} pair $(W^{(i)}_j,W^{(i')}_{j'})$ with $i \neq i'$ is $(\eps,\delta)$-grey.

It follows, by the Embedding Lemma, that there exists a copy of $U(r,k) \subset G$ with exactly one vertex in each set $W^{(i)}_j$. But $B'$ shatters this copy of $U(r,k)$, by condition $(b)$ above, and so $G$ is not $U(r+1,k)$-free, as claimed.
\end{proof}

By Lemma~\ref{Urkfree} and our choice of $k$, it follows that $H_v \le G \in \P$ for some forbidden graph $H_v$, which is a contradiction. Thus $\B(\P_n,\alpha,c) \subset \A(\P_n,\alpha')$, as required.
\end{proof}

Given $\alpha > 0$, a graph $G$ on $n$ vertices, a subset $A \subset V(G)$ and two vertices $u,v \in V(G)$, we say that $u$ is an \emph{$\alpha$-clone} of $v$ with respect to $A$ if
$$\left| \big( \Gamma(u) \cap A \big) \triangle \big( \Gamma(v) \cap A \big) \right| \; \le \; \alpha n.$$
Let $P = (S_1,\ldots,S_r)$ be a partition of $G$, and let $B$ be an $\alpha$-bad set for $(G,P)$. If $B$ is chosen to be maximal, then for every vertex $v \in V(G)$, there exists a vertex $b \in B$, and an index $j \in [r]$ such that $v$ is an $\alpha$-clone of $b$ with respect to $S_j$. Define
$$j(v)  \; := \; \min \big\{ j \in [r] \,:\, \exists \, b \in B \textup{ such that $v$ is an $\alpha$-clone of $b$ with respect to $S_j$} \big\}.$$
Note that the function $j(.)$ in fact depends on the triple $(P,B,\alpha)$. It will usually be obvious which partition $P$, set $B$ and constant $\alpha > 0$ we are using, so we suppress this dependence. When it is not obvious from the context, we shall clarify. 

The following observation is an immediate consequence of the definition of an $\alpha$-clone, together with Observation~\ref{nongrey}.

\begin{obs}\label{choices}
Let $\alpha > 0$ be sufficiently small, let $G$ be a graph on $n$ vertices, let $A,B \subset V(G)$, and suppose $v \in V(G)$ is an $\alpha$-clone of some vertex in $B$ with respect to $A$. Then, given the edges of $G[A,B]$, we have at most
$$|B| 2^{\sqrt{\alpha} n}$$
choices for the edges between $v$ and $A$.
\end{obs}

We would like to have $v \in S_{j(v)}$ for every $v \in V(G)$. In fact we shall prove that, for almost every graph $G \in \P$, we can \emph{adjust} any given BBS-partition of $G$ to guarantee that this holds.

First, we shall show that almost all graphs $G \in \P_n$ have at most $\alpha n/2$ vertices with $v \not\in S_{j(v)}$. Indeed, given $\alpha > 0$, a graph $G$, a BBS-partition $P = (S_1,\ldots,S_r)$ of $G$, and a maximal $\alpha$-bad set $B$ for $(G,P)$, let
$$J(G,P,B,\alpha) \; := \; \{ v \in V(G) \,:\, v \not\in S_{j(v)} \}.$$
Now, given $n \in \N$, $\alpha,\eps,\delta,\gamma > 0$ and a hereditary property of graphs $\P$, let
\begin{align*}
& \C(\P_n,\alpha) \; := \; \Big\{ G \in \P_n \,:\, \exists \textup{ a BBS-partition $P$ of $G$ for $(\eps,\delta,\gamma)$ and a maximal }\\
& \hspace{4cm} \textup{ $(2\alpha)$-bad set $B$ for $(G,P)$ such that $|J(G,P,B,2\alpha)| \ge \alpha n$} \Big\}.
\end{align*}
The next lemma says that the set $\C(\P_n,\alpha)$ is small.

\begin{lemma}\label{countC}
Let $r \in \N$, and let $\P$ be a hereditary property of graphs with $\chi_c(\P) = r$. Let $\alpha > 0$ be small, and let $\eps,\delta,\gamma > 0$ be sufficiently small, and $n \in \N$ be sufficiently large. Then,
$$|\C(\P_n,\alpha)| \; \le \; 2^{(1 - 1/r) n^2/2 \, - \, \alpha n^2/3r^3} \; = \; o(|\P_n|).$$
\end{lemma}

\begin{proof}
The proof is almost the same as that of Lemma~\ref{countA}. Indeed, let $G \in \C(\P_n,\alpha)$, and let $P = (S_1,\ldots,S_r)$ be a BBS-partition of $G$ for $(\eps,\delta,\gamma)$, and $B \subset V(G)$ be a maximal $(2\alpha)$-bad set for $(G,P)$, such that $|J(G,P,B,2\alpha)| \ge \alpha n$.
By the pigeonhole principle, there exists $i,j \in [r]$ (with $i \neq j$) such that
$$|C| := \left| \big\{ v \in S_i \,:\, j(v) = j \big\} \right| \; \ge \; \alpha' n,$$
where $\alpha' = \alpha/r^2$.

Now we simply count the graphs in $\C(\P_n,\alpha)$. We have at most $n^n$ choices for the partition $P$, the set $B$, the index $j$ and the set $C$. By Observation~\ref{nongrey}, and the definition of a BBS-partition, we have at most
$$2^{(\eps + \sqrt{\delta} + \gamma)n^2}$$ choices for the edges inside the set $S_k$, for each $k \in [r]$.

Next, recall that $(1/r - \eps)n \le |S_k| \le (1/r + \eps)n$ for each $k \in [r]$, by the definition of a BBS-partition, and so we have at most
$$2^{(1/r + \eps)^2 n^2}$$
choices for the edges between $S_p$ and $S_q$, for each $p \neq q$. Moreover, we have at most
$$2^{(1/r + \eps)^2 n^2 - (\alpha'/2r) n^2}$$
choices for the edges between $S_i$ and $S_j$. Indeed, by Observation~\ref{choices} we have at most $n 2^{\sqrt{2\alpha'} |C| n}$ choices for the edges between $C$ and $S_j$, and we have at most $2^{(1/r + \eps)^2 n^2 - |C| n/r}$ choices for the edges between $S_i \setminus C$ and $S_j$.

Putting these bounds together, we obtain
$$\log_2 \big( |\A(\P_n,\alpha)| \big) \; \le \; {r \choose 2}(1/r + \eps)^2 n^2 \, - \, \frac{\alpha n^2}{2r^3} \, + \, O\left( (\eps + \sqrt{\delta} + \gamma)n^2 \right).$$
The first inequality now follows if $\eps$, $\delta$ and $\gamma$ are sufficiently small. The final inequality follows by Observation~\ref{ABT}.
\end{proof}

Now let $G$ be a graph, let $r \in \N$, and let $\alpha,\alpha',\eps,\delta,\gamma > 0$. Given a BBS-partition $P = (S_1,\ldots,S_r)$ of $G$ for $(\eps,\delta,\gamma)$, and a maximal $(2\alpha)$-bad set for $(G,P)$, we make the following definition.

\begin{defn}
An \emph{$\alpha$-adjustment} of $(G,P)$ with respect to $B$ is a partition $P' = (S'_1,\ldots,S'_r)$ of $V(G)$ such that, for each $j \in [r]$, the following holds. $|S_j \triangle S'_j| \le \alpha n$, and for every $v \in S'_j$, there exists $b \in B$ such that $v$ is a $(3\alpha)$-clone of $b$ with respect to $S'_j$.
\end{defn}

Given $n \in \N$, constants $\eps,\delta,\gamma,\alpha > 0$, and a hereditary property of graphs $\P$, define
\begin{align*}
&\D(\P_n,\alpha) \; := \; \Big\{ G \in \P_n \,:\, \exists \textup{ a BBS-partition $P$ of $G$ for $(\eps,\delta,\gamma)$ and a maximal } (2\alpha)\textup{-bad}\\
& \hspace{2.2cm} \textup{ set $B$ for $(G,P)$ such that $\nexists$ an $\alpha$-adjustment of $(G,P)$ with respect to $B$} \Big\}.
\end{align*}
The next lemma follows easily from the definitions.

\begin{lemma}\label{countD}
Let $r \in \N$, and let $\P$ be a hereditary property of graphs with $\chi_c(\P) = r$. Let $\alpha > 0$, and let $\eps,\delta,\gamma > 0$ be sufficiently small. Then,
$$\D(\P_n,\alpha) \; \subset \; \C(\P_n,\alpha).$$
\end{lemma}

\begin{proof}
Let $G \in \P_n$, and suppose $G \in \D(\P_n,\alpha) \setminus \C(\P_n,\alpha)$. Let $P = (S_1,\ldots,S_r)$ be an arbitrary BBS-partition of $G$ for $(\eps,\delta,\gamma)$, and let $B$ be a maximal $(2\alpha)$-bad set $B$ for $(G,P)$. Note that, since $G \not\in \C(\P_n,\alpha)$, we have
$$|J(G,P,B,2\alpha)| \; = \; |\{ v \in V(G) \,:\, v \not\in S_{j(v)} \}| \; \le \; \alpha n.$$

For each $j \in [r]$, let $S'_j = \{v \in V(G) \,:\, j(v) = j\}$. We claim that $P' = (S'_1,\ldots,S'_r)$ is an $\alpha$-adjustment of $(G,P)$ with respect to $B$. Indeed, since $|J(G,P,B,2\alpha)| \le \alpha n$, it follows immediately that $|S_j \triangle S'_j| < \alpha n$ for every $j \in [r]$. Moreover, for each $v \in S'_j$ we have $j(v) = j$, and so there exists $b \in B$ such that $v$ is a $(2\alpha)$-clone of $b$ with respect to $S_j$. But $|S_j \triangle S'_j|< \alpha n$, so $v$ is a $(3\alpha)$-clone of $b$ with respect to $S_j$.

Thus $P'$ is an $\alpha$-adjustment of $(G,P)$ with respect to $B$, as claimed. But $P$ and $B$ were chosen arbitrarily, so this contradicts the fact that $G \in \D(\P_n,\alpha)$. Thus $\D(\P_n,\alpha) \; \subset \; \C(\P_n,\alpha)$, as required.
\end{proof}

Finally, for each graph $G$ on vertex set $S_1 \cup \ldots \cup S_r$, and each integer $k \in \N$, we choose a collection of vertex-disjoint copies of $U(t,k)$ for each $2 \le t \le r + 1$, using the following algorithm.
\begin{alg}
Set $\ell := 1$, $t := r + 1$ and $X = \emptyset$. Repeat the following steps until $t = 1$.
\begin{itemize}
\item[1.] Suppose there exists a copy $H$ of $U(t,k)$ in $G - X$, and a function $i : [t] \to [r]$ such that:\\[-2ex]
\begin{itemize}
\item[$(a)$] $V(H) = A_1 \cup \ldots \cup A_t$,\\[-2ex]
\item[$(b)$] $A_{j+1} \to A_1 \cup \ldots \cup A_j$ for each $j \in [t - 1]$,\\[-2ex]
\item[$(c)$] $A_j \subset S_{i(j)} \setminus X$ for each $j \in [t]$, and\\[-2ex]
\item[$(d)$] $i(1) = i(2)$, and $i(j) = i(j') \Leftrightarrow j = j'$ for $j,j' \ge 2$.\\[-2ex]
\end{itemize}
Then set $U_\ell := V(H)$, $X := X \cup U_\ell$ and $\ell := \ell + 1$, and repeat Step 1.
\item[2.] Otherwise, set $t := t - 1$, and go to Step 1.
\end{itemize}
\end{alg}
In other words, we first find a maximal collection of vertex-disjoint copies of $U(r+1,k)$, such that for each copy, the smallest two classes are in the same set ($S_{i(1)} = S_{i(2)}$) as each other, and the other classes in different sets ($S_{i(3)},\ldots,S_{i(r+1)}$). We then find a maximal collection of vertex-disjoint copies of $U(r,k)$, which are also disjoint from each of the copies of $U(r+1,k)$. We repeat this for each $2 \le t \le r + 1$, in decreasing order.

We obtain from the algorithm a collection $\{U_1,\ldots,U_L\}$, where $G[U_\ell] = U(t,k)$ for some $2 \le t \le r + 1$, and the sets $U_\ell$ are pairwise disjoint. The following observation describes the key property of these sets.

\begin{obs}\label{algprop}
Let $U_1, \ldots, U_L$ be the sets obtained from the algorithm applied (for $k$) to the graph $G$ and partition $(S_1,\ldots,S_r)$ of $V(G)$. Then, for each $j \in [r]$ and $\ell \in [L]$, if $U_\ell \cap S_j = \emptyset$ then the set $S_j \setminus \ds\bigcup_{j=1}^\ell U_j$ does not shatter $U_\ell$.
\end{obs}

\begin{proof}
Suppose $G[U_\ell] = U(t,k)$, and $A_{t+1} \to U_\ell$ for some $A_{t+1} \subset S_j \setminus \ds\bigcup_{j=1}^\ell U_j$. Then $U_\ell \cup A_{t+1}$ induces a copy of $U(t+1,k)$, and so this set would have been chosen at an earlier step of the algorithm.
\end{proof}

Now, given $k \in \N$, a graph $G$, a BBS-partition $P$ of $G$ for $(\eps,\delta,\gamma)$, a maximal $(2\alpha)$-bad set $B$ for $(G,P)$, and an $\alpha$-adjustment $P' = (S'_1,\ldots,S'_r)$ of $(G,P)$ with respect to $B$, let
\begin{align*}
& U(G,P',k) \; := \; \bigcup_{\ell = 1}^L U_\ell,
\end{align*}
where $\{U_1,\ldots,U_L\}$ are the sets given by the algorithm, applied to the graph $G$ and the partition $P'$.

Given $n,k \in \N$, constants $\eps,\delta,\gamma,\alpha > 0$, and a hereditary property of graphs $\P$, let
\begin{align*}
\U(\P_n,\alpha,k) \; := \; \Big\{ G \in \P_n \,:\, \exists \textup{ a BBS-partition $P$ of $G$ for $(\eps,\delta,\gamma)$, a maximal $(2\alpha)$-bad}& \\
\hspace{5cm} \textup{set $B$ for $(G,P)$, and an $\alpha$-adjustment $P' = (S'_1,\ldots,S'_r)$ of}& \\
\hspace{6cm} \textup{$(G,P)$ with respect to $B$ with $|U(G,P',k)| \ge n^{1-\alpha}$} \Big\}.&
\end{align*}

Theorem~\ref{structure} is an easy corollary of the following lemma, together with Lemmas~\ref{S1S2}, \ref{countC} and \ref{countD}. The proof of the lemma uses Theorem~\ref{count}, and Lemmas~\ref{Urkfree}, \ref{noUk}, \ref{countA}, \ref{countB}, \ref{countC} and~\ref{countD}.

\begin{lemma}\label{countU}
Let $r \in \N$, and let $\P$ be a hereditary property of graphs with $\chi_c(\P) = r$. There exist $k = k(\P) \in \N$ and $\alpha = \alpha(k,\P) > 0$ such that the following holds. Let $\eps,\delta,\gamma > 0$ be sufficiently small and $n \in \N$ be sufficiently large. Then
$$|\U(\P_n,\alpha,k)| \; \le \; 2^{(1 - 1/r)n^2/2 - \alpha^2 n^{2-\alpha}} \, = \; o(|\P_n|).$$
\end{lemma}

\begin{proof}
Let $k$ be sufficiently large so that $G$ is $U(r+1,k)$-free for every $G \in \P$. Such a $k = k(\P)$ exists by Lemma~\ref{Urkfree}.

We simply count the graphs in
$$\U_n \; := \; \U(\P_n,\alpha,k) \setminus \Big( \B(\P_n,\alpha,n^{1-2\alpha}) \cup \D(\P_n,\alpha) \Big).$$
By Lemmas~\ref{countA} and~\ref{countB} we have
$$|\B(\P_n,\alpha,c)| \; \le \; 2^{(1 - 1/r) n^2/2 \, - \, \alpha' n^2}$$
for some $\alpha' = \alpha'(\P,\alpha)$, if $c = c(\P,\alpha)$ is sufficiently large. (Recall that $|\B(\P_n,\alpha,c)|$ is monotone decreasing in $c$.) Also, by Lemmas~\ref{countC} and~\ref{countD} we have
$$|\D(\P_n,\alpha)| \; \le \; 2^{(1 - 1/r) n^2/2 \, - \, \alpha'' n^2},$$
where $\alpha'' = \alpha/3r^3$. Thus it suffices to prove the claimed bound for the set $\U_n$.

So let $G \in \U_n$, and note that $G$ has
\begin{itemize}
\item[$(a)$] a BBS-partition $P$ for $(\eps,\delta,\gamma)$,
\item[$(b)$] a maximal set $B \subset V(G)$, which is $(2\alpha)$-bad for $(G,P)$, with $|B| \le n^{1-2\alpha}$,
\item[$(c)$] an $\alpha$-adjustment $P' = (S'_1,\ldots,S'_r)$ of $(G,P)$ with respect to $B$, such that
$$|U(G,P',k)| \; \ge \; n^{1-\alpha}.$$
\end{itemize}
Let $U_1,\ldots,U_L$ denote the sets given by the algorithm, applied (for $k$) to the partition $P'$ of $G$. By definition,
$$U(G,P',k) \; = \; \bigcup_j U_j.$$

We have at most $n^n$ choices for the partition $P'$, and the sets $B$ and $U_1,\ldots,U_L$. Now, given an edge $e = ab$, define the index $i(e)$ as follows:
\begin{itemize}
\item[$(a)$] If $e$ has an endpoint in $B$ then $i(e) = 0$.
\item[$(b)$] If $e$ has an endpoint in $U_\ell$ and the other endpoint is in
$$V(G) \setminus \Big( B \cup \bigcup_{j=1}^{\ell-1} U_j \Big)$$
then $i(e) = \ell$.
\item[$(c)$] If $e$ has both endpoints in $V(G) \setminus \Big( B \cup U(G,P',k) \Big)$ then $i(e) = \infty$.
\end{itemize}
We choose the edges of $G$ in increasing order of index.

First, since $|B| \le n^{1-2\alpha}$, we have at most $2^{|B|n} \le 2^{n^{2-2\alpha}}$ choices for the edges incident with $B$. So let $1 \le \ell \le L$, and suppose that $G[U_\ell]$ is a copy of $U(t,k)$, where $2 \le t \le r$. (Note that $t \neq r + 1$, since $G$ is $U(r+1,k)$-free by our choice of $k$.)

\medskip
\noindent \ul{Claim}: There is a constant $\lambda > 0$, depending only on $k$ and $r$, such that we have at most
$$2^{\left( 1 - 1/r - \lambda \right) |U_\ell| n}$$ choices for the edges with index $\ell$.

\begin{proof}
Without loss of generality, let $U_\ell = A_1 \cup \ldots \cup A_t$, where $A_1,A_2 \subset S'_2$, $A_j \subset S'_j$ for each $3 \le j \le t$, and
$$A_{j+1} \to A_1 \cup \ldots \cup A_j$$
for each $1 \le j \le t - 1$. Recall that, by the definition of an $\alpha$-adjustment, each vertex $u \in S'_j$ is a $(3\alpha)$-clone of $b$ with respect to $S'_j$, for some $b \in B$. Note also that $(1/r - 2\alpha)n \le |S'_j| \le (1/r + 2\alpha)n$ for each $j \in [r]$.

Thus, for each $u \in U_\ell$, we have at most
$$2^{\left( 1 - 2/r + 2\sqrt{\alpha} \right) n}$$
choices for the edges between $u$ and $V(G) \setminus S'_1$, by Observation~\ref{choices}, since $\alpha > 0$ is sufficiently small. But $S'_1$ does not shatter $U_\ell$, by Observation~\ref{algprop}, and so, by Lemma~\ref{noUk}, we have at most
$$2^{|U_\ell|n/r - \lambda_1 n}$$
choices for the edges between $U_\ell$ and $S'_1$, where $\lambda_1 = \lambda_1(k,r) > 0$ is the constant in Lemma~\ref{noUk}. Choosing $\alpha = \alpha(\P,k)$ sufficiently small, the result follows.
\end{proof}

Now, let $T_j = S'_j \setminus U(G,P',k)$ for each $j \in [r]$ , and note that, since the algorithm stopped, $T_j$ is $U(k)$-free. Thus, by Theorem~\ref{count}, we have at most $2^{n^{2-2\alpha}}$ choices for the edges inside these sets. Also, trivially, we have at most
$$2^{(1 - 1/r)(n - U(G,P',k))^2/2}$$
choices for the edges between the sets $T_j$.

Multiplying the number of choices, we get
\begin{eqnarray*}
\log_2 ( |\U_n| ) & \le & \left( 1 - \frac{1}{r} \right) \frac{(n - |U(G,P',k)|)^2}{2} \,+\, \left( 1 - \frac{1}{r} - \lambda \right) |U(G,P',k)|n \,+\, O\Big( n^{2-2\alpha} \Big)\\
& \le & \left( 1 - \frac{1}{r} \right) \frac{n^2}{2} \, - \, \lambda |U'(G,P',k)| n \, + \, \frac{|U(G,P',k)|^2}{2} \,+\, O\Big( n^{2-2\alpha} \Big).
\end{eqnarray*}
Since $|U(G,P',k)| \ge n^{1-\alpha}$, the result follows if $|U(G,P',k)| \le \lambda n$.

Finally, suppose that $|U(G,P',k)| \ge \lambda n$. Then, by the pigeonhole principle and without loss of generality, there exists $2 \le t \le r$ and a subset $X \subset [L]$ such that, for each $j \in X$, $U_j = A_1 \cup \ldots \cup A_t$, where $A_1,A_2 \subset S'_2$, $A_j \subset S'_j$ for each $3 \le j \le t$, and
$$A_{j+1} \to A_1 \cup \ldots \cup A_j$$
for each $1 \le j \le t - 1$, and $|X| \ge \lambda_2 n$, where $\lambda_2$ depends only on $k$ and $r$.

We have at most $n^n$ choices for the partitions $(S_1,\ldots,S_r)$ and $(S_1',\ldots,S'_r)$, and at most $2^{(\eps + \sqrt{\delta} + \gamma)n^2}$ choices for the edges inside the sets $S_j$. We have at most $2^{\alpha n^2}$ choices for the edges incident with vertices in $\bigcup_j S_j \triangle S'_j$, and at most $2^{{{r-1} \choose 2}n^2/r^2}$ choices for the edges between $S'_i$ and $S'_j$ for $i,j \neq 1$. Finally, we have at most $2^{n/r(n - n/r) - \lambda_3 n^2}$ choices for the edges incident with $S_1'$, by Lemma~\ref{noUk}.

Thus, choosing $\alpha$ sufficiently small, we obtain
$$\log_2 ( |\U_n| ) \; \le \; \left( \frac{(r-1)(r-2)}{2r^2} + \frac{r-1}{r^2} \right) n^2 - \lambda_4 n^2 \; = \; \left( 1 - \frac{1}{r} - \lambda_4 \right) n^2,$$
for some $\lambda_4 > 0$, as required.
\end{proof}

\begin{rmk}
Note that we in fact only needed $B(G,P,\alpha) \le n^{1-2\alpha}$ for almost every graph $G \in \P$.
\end{rmk}

The proof of Theorem~\ref{structure} now follows easily.

\begin{proof}[Proof of Theorem~\ref{structure}]
Let $\P$ be a hereditary property of graphs with $\chi_c(\P) = r$, and let $k = k(\P) \in \N$ be sufficiently large, and $\alpha = \alpha(\P,k)$, $\gamma = \gamma(\P,\alpha,k)$, $\delta  = \delta(\P,\gamma,\alpha,k) > 0$ and $\eps = \eps(\P,\gamma,\delta,\alpha,k) > 0$ be sufficiently small. By Lemma~\ref{S1S2}, almost every graph $G \in \P$ has a BBS-partition for $(\eps,\delta,\gamma)$. So let $G \in \P$, let $P$ be a BBS-partition of $G$ for $(\eps,\delta,\gamma)$, and let $B$ be a maximal $(2\alpha)$-bad set for $(G,P)$.

Now, by Lemmas~\ref{countC} and~\ref{countD}, for almost every such $G$ there exists an $\alpha$-adjustment $P' = (S'_1,\ldots,S'_r)$ of $(G,P)$ with respect to $B$. Let $U(G,P',k)$ denote the set given by the algorithm. By Lemma~\ref{countU}, $U(G,P',k) \le n^{1-\alpha}$ for almost every such $G$.

Let $A = B \cup U(G,P',k)$, and let $S_j := S'_j \setminus A$ for each $j \in [r]$. Then $S_j$ is $U(k)$-free for each $j \in [r]$, and Theorem~\ref{structure} follows.
\end{proof}

\section{A sharper bound for bipartite graphs}\label{bipsec}

In this section we refine the methods of  the proof of Theorem~\ref{count}, giving a close to sharp upper bound on
$$f(n,n,U(k)) \; := \; |\{G \textup{ bipartite on } A \cup B \,:\, |A| = |B| = n \textup{ and $G$ is $U(k)$-free}\}|.$$

The exponent in our bound will be within a polylog-factor of the best that we could hope for, i.e., the exponent in the extremal result of Alon, Krivelevich and Sudakov~\cite{AKS} (see the more general Theorem~6.1 in their paper), that a $U(k)$-free graph with class sizes $n$ cannot have more than $O(n^{2-1/(k-1)})$ edges.

We remark that to remove all of the seemingly unnecessary $\log$-factors looks extremely hard. Corresponding results are known only for the monotone case, and even then only in special cases (for $C_4$ by Kleitman and Winston~\cite{KWn}, for $C_6$ and $C_8$ by Kleitman and Wilson~\cite{KWl}, for $\{C_4,C_6,\ldots C_{2k}\}$ by Kohayakawa, Kreuger and Steger~\cite{KKS}, and for $K_{s,t}$ by Balogh and Samotij~\cite{BS1,BS2}).

\begin{thm}\label{bipsharp}
For every $3 \le k \in \N$, and every sufficiently large $n$, we have
$$f(n,n,U(k)) \; \le \; \exp\Big( n^{2-1/(k-1)} (\log n)^{k+1} \Big).$$
\end{thm}

We first prove the following lemma, the proof of which uses the methods of Section~\ref{lemmasec}. Let $G$ be a $U(k)$-free bipartite graph with classes $U$ and $V$ where $|U| = |V| = n$. For each $u,v \in U$, we define $\Delta(u,v):=|\Gamma(u) \triangle \Gamma(v)|$, the `distance' between the two vertices.

\begin{lemma}\label{card}
Let $m,n,x,k \in \N$, and let $G$ be a $U(k)$-free bipartite graph with classes $U$ and $V$, where $|U| = m$ and $|V| = n$. Let $U' \subset U$, and suppose that, for any $u,v \in U'$, we have $\Delta(u,v) \ge x$. Then
$$|U'| \; \le \; \left(\frac{n}{x}\right)^{k-1} 3^k \big( \log m \big)^{k-1}.$$
\end{lemma}

\begin{proof}
The proof is very similar to that of Lemma~\ref{key}. Indeed, let $c = |U'|$, and with foresight, let $p = \ds\frac{3 \log c}{x}$. We claim that there exists a subset $X \subset V$, with $|X| = pn$, such that
$$\Gamma(u) \cap X \; \neq \; \Gamma(v) \cap X$$
for each $u,v \in U'$ with $u \neq v$. Indeed, if we choose a random subset $X \subset V$ of size $pn$, then
$$\Pr\Big( \Gamma(x) \cap X = \Gamma(y) \cap X \Big) \; \le \; {{n - x} \choose {pn}} \Big/ {n \choose {pn}} \; \le \; \left( 1 - \frac{x}{n} \right)^{p n}$$
for each such pair $\{x,y\}$, and so
$$\Ex \left| \left\{ \{x,y\} \in {{U'} \choose 2} \,:\, \Gamma(x) \cap X = \Gamma(y) \cap X \right\} \right| \; \le \; {c \choose 2} \left( 1 - \frac{x}{n} \right)^{pn} \; \le \; c^2 e^{-px} \; < \; 1.$$
Thus such a set $X$ must exist, as claimed. Now, if
$$c \; > \; 2 \left( \frac{3n \log c}{x} \right)^{k-1} \; = \; 2(pn)^{k-1} \; \ge \; \sum_{i=0}^{k-1} {{|X|} \choose i},$$
then, by Sauer's Lemma, there exist sets $U'' \subset U'$ and $X' \subset X$, with $|X'| = k$, such that $U'' \to X'$. But this is a contradiction, since $G$ is $U(k)$-free. Thus
$$|U'| \; = \; c \; \le \; 2 \left( \frac{3n \log c}{x} \right)^{k-1} \; \le \; 3^k \left( \frac{n}{x} \right)^{k-1} \big( \log m \big)^{k-1}$$
as required.
\end{proof}

We are now ready to prove Theorem~\ref{bipsharp}.

\begin{proof}[Proof of Theorem~\ref{bipsharp}]
Let $|U| = |V| = n$, and suppose that $G$ is a bipartite $U(k)$-free graph with classes $U$ and $V$. We are required to show that the number of choices for the edge set of $G$ is at most $\exp\Big( n^{2-1/(k-1)} (\log n)^{k+1} \Big)$.

The idea is to partition $U$ into $t + 2$ parts, and consider the edges from each part to $V$ in turn. Indeed, let $t \in \N$, let $n > x_0 > \dots > x_t > n^{1 - 1/(k-1)}$, and let
$$U_0 \,\subset \,\dots \,\subset\, U_t \,\subset \, U$$
be maximal subsets satisfying $\Delta(u,v) \ge x_i$ for each $u,v \in U_i$ and $0 \le i \le t$. Moreover, and with foresight, let $t = 10\log\log n$, and let
$$x_i \; := \; n^{1-1/(k-1)+1/(k-1)^{i+2}}$$
for each $0 \le i \le t$. Note that we have at most $2^n$ choices for the sets $U_i$.

Now, by Lemma~\ref{card},
$$|U_i| \; \le \; \left(\frac{n}{x_i}\right)^{k-1} 3^k (\log n)^{k-1}$$
for each $0 \le i \le t$, so we have at most
$$\exp\left( n \left( \frac{n}{x_0} \right)^{k-1} 3^k (\log n)^{k-1} \right)$$
choices for the graph $G[U_0,V]$. Now let $0 \le i \le t - 1$, and assume that the sets $U_i$ and $U_{i+1}$ and the graph $G[U_i,V]$ have already been chosen.

\medskip
\noindent \ul{Claim}: We have at most
$$\exp\left( x_i \left( \frac{n}{x_{i+1}} \right)^{k-1} 3^{k+1} (\log n)^k \right)$$
choices for the edges between $U_{i+1} \setminus U_i$ and $V$.

\begin{proof}[Proof of claim]
Since $U_i$ is maximal, for every $v \in U_{i+1} \setminus U_i$ there is a $u \in U_i$ such that $\Delta(u,v) < x_i$. Thus, the number of choices for the edges between $v$ and $V$ is at most
$$2|U_i|{n \choose x_i},$$
and so the number of choices for the graph $G[U_{i+1} \setminus U_i,V]$ is at most
$$\left( 2|U_i| {n \choose {x_i}} \right)^{|U_{i+1}|} \; \le \; \big( n^{x_i + 1} \big)^{(n/{x_{i+1}})^{k-1} 3^k (\log n)^{k-1} } \; \le \; \exp\left( x_i \left( \frac{n}{x_{i+1}} \right)^{k-1} 3^{k+1} (\log n)^k \right),$$
as claimed.
\end{proof}

Finally, given $U_t$ and $G[U_t,V]$, the number of choices for the edge between $U \setminus U_t$ and $V$ is at most
$$\left[ 2 |U_t| {n \choose {x_t}} \right]^n \; \le \; \exp\Big((x_t + 1) n \log n \Big) \; \le \; \exp\Big( 2 n^{2-1/(k-1)} \log n \Big).$$
since for each $u$ there is a vertex $v\in U_t$ such that $\Delta(u,v)\le x_t$. There are at most $n$ choices for $v$ and $2{n \choose {x_t}}$ choices for the symmetric difference. 

Putting these bounds together, the number of choices for the graph $G[U,V]$ is at most
\begin{align*}
& \exp\left( n \left( \frac{n}{x_0} \right)^{k-1} 3^k (\log n)^{k-1} \,+ \, 2 n^{2-1/(k-1)} \log n  \, + \sum_{i=0}^{t-1} x_i \left( \frac{n}{x_{i+1}} \right)^{k-1} 3^{k+1} (\log n)^k \right).
\end{align*}
But, recalling that $x_i = n^{1-1/(k-1)+1/(k-1)^{i+2}}$, we have
$$n \left( \frac{n}{x_0} \right)^{k-1} \; = \; x_i \left( \frac{n}{x_{i+1}} \right)^{k-1} \; = \; n^{2 - 1/(k-1)}$$
for each $0 \le i \le t - 1$. Since $t = O(\log\log n)$, this gives an upper bound of
$$\exp\Big( C n^{2-1/(k-1)}(\log n)^k \log\log n \Big),$$
as required.
\end{proof}

We cannot hope to obtain very sharp results from such a (relatively) simple application of Sauer's Lemma. However, our results are close to optimal, if we do not care about poly-log factors in the exponent. For example, Theorem~\ref{bipsharp} for $k = 3$ gives an $n^{3/2}$ in the exponent, and $3/2$ is best possible, as $U(3)$ contains a $C_4$. In general, $U(k)$ contains a graph with average degree at least $2(k - 3)$, and so we have the bounds
$$\exp \Big( n^{2 - 1/(k-3)} \Big) \; \le \; f\big( n,n,U(k) \big) \; \le \; \exp\Big( n^{2 - 1/(k - 1)} (\log n)^{k + 1} \Big).$$

It would be interesting to eliminate (if possible) the $\log n$-factors from the exponent; this would yield Kleitman-Winston-type results for many different bipartite graphs.

\section{questions}\label{qsec}

The most obvious disadvantage of Theorem~\ref{structure} is that we know almost nothing about the structure of a typical $U(k)$-free graph.

\begin{qu}\label{sparse}
What is the structure of a typical $U(k)$-free graph? In particular, are almost all $U(k)$-free graphs either dense or sparse?
\end{qu}

As we remarked in the Introduction, there has recently been some important progress on hereditary properties of hypergraphs. In particular, we noted the following theorems of Dotson and Nagle~\cite{DN}, Ishigami~\cite{Ish} and Person and Schacht~\cite{PSch}.

Given a hereditary property of $k$-uniform hypergraphs $\P$, the extremal number of $\P$ is defined to be
\begin{align*}
& \textbf{ex}(n,\P) \; := \; \max\bigg\{ |\A| : \A \subset {[n] \choose k}, \textup{ and there exists } \M \subset {[n] \choose k} \setminus \HH \textup{ such that } \\
& \hspace{10cm} \M \cup \A' \in \P  \textup{ for every } \A' \subset \A \bigg\}.
\end{align*}
In other words, it is the maximum dimension of a subspace of $\P_n$, in the product space $\{0,1\}^{{[n] \choose k}}$.

\begin{thm}[Dotson and Nagle~\cite{DN}, Ishigami~\cite{Ish}]
Let $k \in \N$ and let $\P$ be a hereditary property of $k$-uniform hypergraphs. Then
$$|\P_n| \; = \; 2^{\mathbf{ex}(n,\P) + o(n^k)}.$$
\end{thm}

\begin{thm}[Person and Schacht~\cite{PSch}]
Almost every Fano-plane-free 3-uniform hypergraph is bipartite.
\end{thm}

These results give reason to be optimistic that the following question, which until recently would have seemed very far out of reach, may now be approachable.

\begin{qu}
What is the structure of a typical member of a hereditary property of $k$-uniform hypergraphs?
\end{qu}

Finally, we note that Theorem~\ref{structure} is considerably weaker than Conjecture~\ref{herconj}, since the set $A$ can be very large, and because our bounds on $k$ are likely far from best possible. The following problem asks for some progress towards the conjecture.

\begin{prob}
In the statement of Theorem~\ref{structure}, improve the upper bound on $|A|$, and give good bounds on the constant $k(\P)$.
\end{prob}

\end{document}